\documentclass{article}%
\usepackage{amsfonts}
\usepackage{amsmath}
\usepackage{amssymb}
\usepackage{graphicx}%
\providecommand{\U}[1]{\protect\rule{.1in}{.1in}}
\newtheorem{theorem}{Theorem} [section]

\newtheorem{claim}[theorem]{Claim}
\newtheorem{conclusion}[theorem]{Conclusion}

\newtheorem{corollary}[theorem]{Corollary}

\newtheorem{definition}[theorem]{Definition}

\newtheorem{proposition}[theorem]{Proposition}
\newtheorem{remark}[theorem]{Remark}

\newenvironment{proof}[1][Proof]{\noindent\textbf{#1.} }{\ \rule{0.5em}{0.5em}}
\begin{document}
\author{Robert Shwartz\\Department of Mathematics\\Ariel University, Israel\\robertsh@ariel.ac.il
\and Hadas Yadayi\\Department of Mathematics (graduated)\\Ariel University, Israel\\hadas111096@gmail.com}
\date{}
\title{Sequences over finite fields defined by OGS and BN-pair decompositions of $PSL_{2}(q)$ recursively}
\maketitle
\begin{abstract}
\noindent Factorization of groups into Zappa-Sz\'{e}p product, or more generally into $k$-fold Zappa-Sz\'{e}p product of its subgroups, is an interesting problem, since it eases the multiplication of two elements in a group, and has recently been applied for public-key cryptography as well. We give a generalization of the $k$-fold Zappa-Sz\'{e}p product of cyclic groups, which we call $OGS$ decomposition. It is easy to see that existence of an $OGS$ decomposition for all the composition factors of a non-abelian group $G$ implies the existence of an $OGS$ for $G$ itself. Since the composition factors of a soluble group are cyclic groups, it obviously has an $OGS$ decomposition. Therefore, the question of the existence of an $OGS$ decomposition is interesting for non-soluble groups.  The Jordan-H\"{o}lder Theorem motivates us to consider an existence of an $OGS$ decomposition for the finite simple groups. In 1993, Holt and Rowley showed that $PSL_{2}(q)$ and $PSL_{3}(q)$ can be expressed as a product of cyclic groups.  In this paper, we consider an $OGS$ decomposition of $PSL_{2}(q)$ from point of view different than that of Holt and Rowley. We look at its connection to the $BN-pair$ decomposition of the group.  This connection leads to sequences over $\mathbb{F}_{q}$, which can be defined recursively, with very interesting properties, and  which are closely connected to the Dickson and to the Chebyshev polynomials. Since every finite simple group of Lie-type has $BN-pair$ decomposition, the ideas of the paper might be generalized to further simple groups of Lie-type.



\end{abstract}

\section{Introduction}

The fundamental theorem of finitely generated abelian groups states the following:
Let $A$ be a finitely generated abelian group, then there exists generators $a_{1}, a_{2}, \ldots a_{n}$, such that every element $a$ in $A$ has a unique presentation of a form:
$$g=a_{1}^{i_{1}}\cdot a_{2}^{i_{2}}\cdots a_{n}^{i_{n}},$$
where, $i_{1}, i_{2}, \ldots, i_{n}$ are $n$ integers such that for  $1\leq k\leq n$, $0\leq i_{k}<|g_{k}|$, where $a_{k}$ has a finite order of $|a_{k}|$ in $A$, and $i_{k}\in \mathbb{Z}$, where $a_{k}$ has infinite order in $A$.
Where, the meaning of the theorem is that every abelian group $A$ is direct sum of finitely many cyclic subgroup $A_{i}$ (where $1\leq i\leq k$), for some $k\in \mathbb{N}$.

The mentioned  property of abelian groups yields a natural question whether a finite group (not necessarily abelian) can be factorized into it's subgroups.
\begin{definition}\cite{miller, zappa, szep}\label{zszep}
 Let $G$ be a group, and let $H$ and $K$ its subgroups, such that the following holds:
 \begin{itemize}
 \item $|G|=|H|\cdot |K|$;
 \item $H\cap K=\{1\}$;
 \item $G=HK$.
 \end{itemize}
Then obviously, every element $g\in G$ has a unique presentation of the form $g=hk$, such that $h\in H$ and $k\in K$, and $G$ is called  the Zappa-Sz\'{e}p product of its subgroups $H$ and $K$, and is denoted $G=H\bowtie K$.
\end{definition}
The Zappa-Sz\'{e}p product of $H$ and $K$ involves appropriate actions of $H$ on $K$ and of $K$ on $H$, which allow easy multiplication of elements in pair form. Therefore, a lot of work has been carried out about factorization of a group into a Zappa-Sz\'{e}p product of its subgroups, and the literature contains several hundreds of papers in the last 30 years about factorization of simple groups by Walls, Praeger, Jones, Heng Li, Giudici, Wiegold, Williamson, Baumeister, Liebeck, Saxl, and many others. In 1984, Arad and Fishman \cite{AF} completely classified which simple groups $G$ can be written as the Zappa-Sz\'{e}p product $G=A\bowtie B$.
The factorization of a group $G$ into the Zappa-Sz\'{e}p product can be generalized to a $k$-fold Zappa-Sz\'{e}p product for subgroups $H_{1}, H_{2}, \ldots, H_{k}$, where the product of the $k$-subgroups $H_{1}H_{2}\cdots H_{k}$ represents uniquely every element in $G$. For example, $\mathbb{Z}_{2}\bigoplus \mathbb{Z}_{2}\bigoplus \mathbb{Z}_{2}$ can be considered as a $3$-fold Zappa-Sz\'{e}p product for subgroups $H_{1}$, $H_{2}$, and $H_{3}$, where $H_{i}\simeq \mathbb{Z}_{2}$ for $1\leq i\leq 3$.  In 1928, Hall \cite{Hall} proved that every soluble group is a product of its Sylow subgroups. By \cite{HR}, the simple groups $PSL_{2}(q)$, $PSL_{3}(q)$, and some alternating groups can be expressed as the product of their Sylow subgroup as well.  In 2003, Vasco, R\"{o}tteler, and Steinweidt \cite{VRS} proved that the five sporadic Mathieu groups ($M_{11}$, $M_{12}$, $M_{22}$, $M_{23}$, and $M_{24}$) are products of their Sylow subgroups as well, and they motivate the factorizations to public-key cryptography.

There is a special interest in the $k$-fold Zappa-Sz\'{e}p product for cyclic subgroups, which generalize the fundamental theorem of finitely generated abelian groups for the non-abelian case. The problem of factoring a group $G$ into a Zappa-Sz\'{e}p product of two cyclic subgroups was posed by Ore in 1937 \cite{Ore}. It was considered in a series of papers by Jesse Douglas in the early 1950s \cite{Douglas1, Douglas2, Douglas3, Douglas4}. However the problem is still open.
In \cite{SAR}, an interesting  generalization of it, which is called $OGS$, was proposed.  We start by recalling the definition of the $OGS$ decomposition.
\begin{definition}\label{ogs}
Let $G$ be a non-abelian group. The ordered sequence of $n$ elements $\langle g_{1}, g_{2}, \ldots, g_{n}\rangle$ is called an $Ordered ~~ Generating ~~ System$ of the group $G$ or by shortened notation, $OGS(G)$, if every element $g\in G$ has a unique presentation in the form
$$g=g_{1}^{i_{1}}\cdot g_{2}^{i_{2}}\cdots g_{n}^{i_{n}},$$
where, $i_{1}, i_{2}, \ldots, i_{n}$ are $n$ integers such that for  $1\leq k\leq n$, $0\leq i_{k}<r_{k}$, where  $r_{k} | |g_{k}|$  in case the order of $g_{k}$ is finite in $G$, or   $i_{k}\in \mathbb{Z}$, in case  $g_{k}$ has infinite order in $G$.
\end{definition}

For example, the quaternion group $Q_{8}$ has an $OGS$ decomposition, although $Q_{8}$ can not be expressed as a Zappa-Sz\'{e}p product of its cyclic subgroups.  Let $$Q_8=\langle a, b | a^4=1, a^2=b^2, bab^{-1}=a^{-1}\rangle$$
Then, every element of $Q_{8}$ has a unique presentation in a form $a^{i}\cdot b^{j}$, where $0\leq i<4$, and $0\leq j<2$ (although $|b|=4$). Considering  the defining relation $ba=a^{-1}b$, we can easily multiply two elements of the form $$x_{1}=a^{i_1}\cdot b^{j_1},  ~~x_{2}=a^{i_2}\cdot b^{j_2}.$$
\begin{itemize}
\item If $j_1=0$, then $x_{1}\cdot x_{2}=a^{i_1+i_2}\cdot b^{j_1+j_2}$;
\item If $j_1=1, ~~j_2=0$, then $x_{1}\cdot x_{2}=a^{i_1-i_2}\cdot b^{j_1+j_2}$;
\item If $j_1=1, ~~j_2=1$, then $x_{1}\cdot x_{2}=a^{i_1-i_2+2}$.
\end{itemize}

Another example is the alternating group $Alt_{6}$. By \cite{miller} the following holds:
 \begin{itemize}
\item  $Alt_{2n+1}$ is isomorphic to a Zappa-Sz\'{e}p product of $Alt_{2n}$ and $\mathbb{Z}_{2n+1}$;
\item $Alt_{4n}$ is isomorphic to a Zappa-Sz\'{e}p product of $Alt_{4n-1}$ and $Dih_{2n}$ (the dihedral group $\mathbb{Z}_{2n}\rtimes \mathbb{Z}_{2}$);
\item $Alt_{6}$ can not be expressed as a Zappa-Sz\'{e}p product of its subgroups.
\end{itemize}
Although by \cite{salt} there exists an $OGS$ for $Alt_{n}$ for every $n$ (including $n=6$), with very interesting multiplication laws.

The last examples (quaternion group $Q_{8}$, alternating group $Alt_{6}$) demonstrate that there are non-abelian groups with an $OGS$ decomposition which does not coincide with any  $k$-fold Zappa-Sz\'{e}p product. Therefore, the requirement  $r_{k} | |g_{k}|$ (and not necessarily $r_{k}=|g_{k}|$)  is essential in Definition \ref{ogs}, which generalizes the definition of  $k$-fold Zappa-Sz\'{e}p product of cyclic groups.

Similarly to  Zappa-Sz\'{e}p products, the $OGS$ decomposition of a group $G$ can ease the multiplication of two elements in $G$. Moreover, the definition of $OGS$ decomposition for the classical Coxeter groups \cite{BB} has close connections to the Coxeter length function \cite{Sh} and more generally, in the case of the complex reflection groups \cite{STo}, the $OGS$ decomposition is closely connected to the dimensions of components of an algebra such that the complex reflection group acts on it (for details, see \cite{SAR}). Therefore, it is interesting to see whether a generic non-abelian group has an $OGS$ decomposition.

The question of the existence of an $OGS$ decomposition for a finite group $G$ can be reduced to the existence of an $OGS$ decomposition of all the composition factors of it. Therefore, now we mention the composition factors of a finite group, and its connection to an $OGS$ decomposition of it.

Let $G$ be a finite group. Consider the series
$$G=G_{0}\trianglerighteq G_{1}\trianglerighteq G_{2}\trianglerighteq\cdots \trianglerighteq G_{n}=1,$$
where, $G_{i}$ is a maximal normal subgroup of $G_{i-1}$ for every $1\leq i\leq n$.
Then \\ $K_{i}=G_{i-1}/G_{i}$ is a finite simple group for $1\leq i\leq n$, and $$|G|=|K_{1}|\cdot |K_{2}|\cdots |K_{n}|.$$
The Jordan-H\"{o}lder theorem states that the multi-set $\{K_{1}, K_{2}, \ldots K_{n}\}$ is an invariant of the group $G$, which does not depend on the choice of the maximal subgroup $G_{i}$ \cite{J1, J2, H, S}. Thus, $K_{1}, K_{2}, \ldots, K_{n}$ are called the composition factors of the group $G$.
If $K_{i}$ has an $OGS$ for every $1\leq i\leq n$, then by taking the elements of $G$ which are the corresponding representatives to the cosets in
the sub-quotients $K_{i}$,  we get an $OGS$ for $G$.
A group $G$ is called soluble if all composition factors of the group are cyclic. Thus, obviously, every finite soluble group $G$ has an $OGS$, which are the coset representatives  corresponding to the generators of the cyclic sub-quotients $K_{1}, K_{2}, \ldots, K_{n}$.
Therefore, it motivates us to check whether a non-soluble group has an $OGS$ presentation. Since all the composition factors of a finite group $G$ are simple groups, we can reduce the question of the existence of an $OGS$ decomposition to simple groups only. The classification of finite simple groups was announced in 1981 by Gorenstein  \cite{G}, and was completed in 2004 by Aschbacher and Smith \cite{AS}. The classification states that there are three categories of non-abelian simple groups:
\begin{itemize}
\item The alternating group $Alt_{n}$ for $n\geq 5$ (The even permutations);
\item The simple groups of Lie-type;
\item 26 sporadic finite simple groups.
\end{itemize}
Since the most significant category of simple groups is the simple groups of Lie-type, it is interesting enough to try to answer the question of whether a finite simple  group of Lie-type has an $OGS$ presentation.

\subsection{$BN-pair$ decomposition}

The simple groups of Lie-type can be considered by  matrix presentation. Moreover, J. Tits introduced  the $BN-pair$ decomposition \cite{T1, T2}  for every simple group of Lie-type, i.e., there exist specific subgroups $B$ and $N$, such that the simple group of Lie-type is the product $BNB$, with the properties which we now describe. We use the notation of \cite{St}.

\begin{definition}\label{bnpair}
Let $G$ be a group, then the subgroups $B$ and $N$ of $G$ form a $BN-pair$ if the following axioms are satisfied:
\begin{enumerate}
\item $G=\langle B, N\rangle$;
\item $H=B\cap N$ is normal in $N$;
\item $W=N/H$ is generated by a set $S$ of involutions, and $(W,S)$ is a Coxeter system;
\item If $\dot{w}\in N$ maps to $w\in W$ and $\dot{s}\in N$ maps to $s\in S$  (under $N\rightarrow W$) , then for every $\dot{v}$ and $\dot{s}$:
 \begin{itemize}
 \item $\dot{s}B\dot{w}\subseteq B\dot{s}\dot{w}B\cup B\dot{w}B$;
 \item $\dot{s}B\dot{s}\neq B$.
\end{itemize}
\end{enumerate}
\end{definition}

\begin{definition} \label{splitbnpair}
Let $G$ be a group with a $BN-pair$, $(B,N)$. It is a split $BN$-pair with a
characteristic $p$, if the following additional hypotheses are satisfied:
\begin{enumerate}
\item $B=UH$, with $U=O_p(B)$, the largest normal $p$-subgroup of $B$, and $H$ a complement of $U$;
\item $\cap_{n\in N}nBn^{-1}=H$.
\end{enumerate}

\end{definition}

The finite simple groups of Lie-type has a split $BN-pair$ with a characteristic $p$, where $p$ is the characteristic of the field over the matrix group is defined. Consider the easiest case of a simple group of Lie-type, which is $PSL_{2}(q)$. Although by \cite{HR} there is an $OGS$ for $PSL_{2}(q)$ and for $PSL_{3}(q)$, in this paper we look at it from a different point of view, connecting the $OGS$ to the $BN-pair$ decomposition. That connection gives motivation for generalizing it to further simple groups of Lie-type, even those which are not mentioned in \cite{HR}.  Moreover, the connection yields interesting recursive sequences over finite fields, which involve the Dickson polynomials of the second kind \cite{D} as well. Since the Dickson polynomial of the second kind is closely connected to the Chebyshev polynomial of the second kind \cite{rivlin}, namely $U_{n}(x)=E_{n}(2x,1)$ (where $U_{n}$ is the Chebyshev polynomial of the second kind, and $E_{n}$ is the Dickson polynomial of the second kind), studying the recursive sequences which are connected to the $OGS$ and the $BN-pair$ decomposition of $PSL_{2}(q)$ helps us to understand interesting properties of the Chebyshev polynomials as well.  Now, we look at $PSL_{2}(q)$, then the following properties are satisfied:

\begin{claim}\label{bn-matrix}
Let $G=PSL_{2}(q)$, and let $B, N, H, U, S$ be the groups as it is defined in Definitions \ref{bnpair} and \ref{splitbnpair}, then the following observations hold:
\begin{itemize}
\item $B$ is the subgroup in which coset representatives are the upper triangular matrices in $SL_{2}(q)$;
\item $N$ is the subgroup in which coset representatives are the monomial matrices in $SL_{2}(q)$  (i.e., there is exactly one non-zero entry in each row and in each column);
\item $H=B\cap N$ is the subgroup in which coset representatives are the diagonal matrices  in $SL_{2}(q)$. Thus, $H$ is isomorphic to a quotient of the multiplicative group of $\mathbb{F}_{q}^{*}$. Thus, denote the elements of $H$ by $h(y)$ where $y\in \mathbb{F}_{q}^{*}$;
\item $S=N/H$ is isomorphic to $\mathbb{Z}_2$. Denote by $s$ the element of $PSL_{2}(q)$, which corresponds to the non-trivial element of $S$;
\item $U$ is the subgroup in which coset representatives are the upper unipotent matrices in $SL_{2}(q)$ (i.e., the upper triangular matrices with diagonal entry $1$). Thus, $U$ is isomorphic to the additive group of $\mathbb{F}_{q}$. Thus, denote the elements of $U$ by $u(x)$ where $x\in \mathbb{F}_{q}$.
\end{itemize}
\end{claim}

By Definition \ref{bnpair}, $PSL_{2}(q)$ is generated by $u(x)$, $h(y)$ and $s$, where  $x\in \mathbb{F}_{q}$, and $y\in \mathbb{F}_{q}^{*}$, and the following proposition holds:

\begin{proposition}\label{bn-canon-psl}
Let $G=PSL_{2}(q)$, where $\mathbb{F}_{q}$ is a finite field, then the $BN-pair$ presentation of an element $g\in G$ is as follows:
\begin{itemize}
\item If $g\notin B$, then $g=u(a)\cdot s\cdot u(x)\cdot h(y)$, where $a,x\in \mathbb{F}_{q}$ and $y\in \mathbb{F}_{q}^{*}$;
\item If $g\in B$, then $g=u(x)\cdot h(y)$, where $x\in \mathbb{F}_{q}$ and $y\in \mathbb{F}_{q}^{*}$.
\end{itemize}
\end{proposition}

The proof is an immediate conclusion of the properties of $U$, $H$, and $S$ which have been described in Definitions \ref{bnpair}, \ref{splitbnpair}. One can easily conclude from Proposition \ref{bn-canon-psl} the following property:

\begin{conclusion}\label{coset-bn-unique}
Let $G=PSL_{2}(q)$, where $\mathbb{F}_{q}$ is a finite field, let $g_1, g_2$ be two elements of $G$, where
$$g_{1}=u(a_{1})\cdot s\cdot u(x_{1})\cdot h(y_{1})$$
and
$$g_{2}=u(a_{2})\cdot s\cdot u(x_{2})\cdot h(y_{2}).$$
Then,
$$a_{1}=a_{2}$$
if and only if
$$g_{1}B=g_{2}B ~~(i.e., ~g_{2}^{-1}g_{1}\in B).$$
\end{conclusion}

\begin{definition}\label{bn-presentation-matrix}
Let $G=PSL_{2}(q)$, $x\in \mathbb{F}_{q}$, $y\in \mathbb{F}_{q}^{*}$, let $u(x)$, $h(y)$, and $s$ be as it is defined in Claim \ref{bn-matrix}, and let $\hat{u}(x)$, $\hat{h}(y)$, and $\hat{s}$ be the corresponding coset representatives in $SL_{2}(q)$. Then we can choose the representative matrices as follows:
\begin{itemize}
\item $\hat{u}(x)=\begin{pmatrix} 1 & x \\
0 & 1\end{pmatrix}$;
\item $\hat{h}(y)=\begin{pmatrix} y & 0 \\
0 & y^{-1}\end{pmatrix}$;
\item $\hat{s}=\begin{pmatrix} 0 & 1 \\
-1 & 0\end{pmatrix}$.
\end{itemize}
\end{definition}

Now, we easily conclude the following multiplication laws between $u(x)$, $h(y)$, and $s$.

\begin{proposition}\label{uhs-multiply}
Let $G=PSL_{2}(q)$, $x\in \mathbb{F}_{q}$, $y\in \mathbb{F}_{q}^{*}$, let $u(x)$, $h(y)$, and $s$ be as it is defined in Claim \ref{bn-matrix}. Then the following relations hold:
\begin{itemize}
\item $u(x_1)\cdot u(x_2)=u(x_1+x_2)$;
\item $h(y_1)\cdot h(y_2)=h(y_1\cdot y_2)$;
\item $h(y)\cdot u(x)=u(x\cdot y^{2})\cdot h(y)$;
\item $s\cdot u(x)\cdot s=u(-x^{-1})\cdot s\cdot u(-x)\cdot h(x)$;
\item $s\cdot h(y)=h(y^{-1})\cdot s$.
\end{itemize}

\end{proposition}

The proof comes directly from the properties of $2\times 2$ matrix multiplications.

\subsection{Dickson polynomials and related polynomials over finite fields}

Now, we recall the definition of the Dickson polynomials  of the second kind \cite{D} over a finite field $\mathbb{F}_{q}$, we show some important properties of it, and we define some other polynomials over $\mathbb{F}_{q}$, which are closely related to Dickson polynomials, and are used in the paper.

\begin{definition}\label{alpha}
For  $a\in \mathbb{F}_{q}$, define $\alpha_{-1}(a)$ to be $0$, and for $k\geq 0$, define $\alpha_{k}(a)$ to be $E_{k}(a,1)$, the Dickson polynomial of the second kind of degree $k$ \cite{D} over $\mathbb{F}_{q}$, on variable $a$ and with parameter $1$, i.e.:
\begin{itemize}
\item For $k=2r$: $$\alpha_{2r}(a)=E_{2r}(a,1)=\sum_{i=0}^{r}(-1)^{i}{2r-i \choose i}a^{2r-2i}$$
\item For $k=2r+1$: $$\alpha_{2r+1}(a)=E_{2r+1}(a,1)=\sum_{i=0}^{r}(-1)^{i}{2r-i+1 \choose i}a^{2r-2i+1}.$$
\end{itemize}
\end{definition}

\begin{remark}\label{chebyshev-alpha}
In case of odd $q$, the Dickson polynomial of the second kind on variable $a$ and with parameter $1$ is the same as the Chebyshev polynomial of the second kind over $\mathbb{F}_{q}$ on variable $\frac{a}{2}$ \cite{rivlin}.
\end{remark}

\begin{proposition}\label{alpha-recursive}
For  $a\in \mathbb{F}_{q}$ and $k\geq 0$,  the following holds:
$$\alpha_{k+1}(a)=a\cdot \alpha_{k}(a)-\alpha_{k-1}(a).$$
\end{proposition}

\begin{proof}
The proof follows directly from the properties of the Dickson polynomials.
\end{proof}
\begin{definition}\label{beta}
For $a, b\in \mathbb{F}_{q}$ and  $k\geq -1$, define $\beta_{k}(a,b)$ and $\gamma_{r}(a,b)$ as follow
\begin{itemize}
\item $\beta_{-1}(a,b)=1$;
\item $\beta_{k}(a,b)=b\cdot\alpha_{k}(a)-\alpha_{k-1}(a)$, ~where  $k\geq 0$;
\item $\gamma_{k}(a,b)=\alpha_{k}(a)+b\cdot\beta_{k}(a,b)$.
\end{itemize}
\end{definition}

\begin{remark}\label{beta-gamma-linear}
Since for every $k\geq 1$, both $\beta_{k}(a,b)$ and $\gamma_{k}(a,b)$ are linear combinations of elements of the form $\alpha_{r}(a)$,  we have the following recurrence relations
$$\beta_{k+1}(a,b)=a\cdot \beta_{k}(a,b)-\beta_{k-1}(a,b), ~~~~\gamma_{k+1}(a,b)=a\cdot \gamma_{k}(a,b)-\gamma_{k-1}(a,b).$$
\end{remark}

\begin{proposition}\label{chebyshev}
Let $q=p^{n}$ such that $p$ is an odd prime, and for every $k\geq 0$, $a, b\in \mathbb{F}(q)$, let $\alpha_{k}(a)$ and $\beta_{k}(a,b)$ as it is defined in Definitions \ref{alpha}, \ref{beta}. Assume $b=\frac{a}{2}$. Then for every $k\geq 0$ the following holds:
\begin{itemize}
\item $\beta_{k-1}(a,b)=T_{k}(b)$, where $T_{k}(b)$ is the Chebyshev polynomial of the first kind of degree $k$  on variable $b$ over $\mathbb{F}_{q}$;
\item $\alpha_{k}(a)=U_{k}(b)$, where $U_{k}(b)$ is the Chebyshev polynomial of the second kind of degree $k$ on variable $b$ over $\mathbb{F}_{q}$.
\end{itemize}
\end{proposition}

\begin{proof}
By Definition \ref{beta}, $\beta_{-1}(a,b)=1$ and $\beta_{0}(a,b)=b$. Now, since $a=2b$, by applying Remark \ref{beta-gamma-linear}, we have $\beta_{k+1}(a,b)=2b\cdot \beta_{k}(a,b)-\beta_{k-1}(a,b)$, which implies $\beta_{k-1}(a,b)=T_{k}(b)$ for every $k\geq -1$.
Since $b=\frac{a}{2}$, we have $\alpha_{k}(a)=U_{k}(b)$ by immediate application of Remark \ref{chebyshev-alpha}.
\end{proof}

\begin{proposition}\label{fibonacci}
Let $q=p^{n}$ such that $p$ is an odd prime, and for every $k\geq 0$, $a, b\in \mathbb{F}(q)$, let $\alpha_{k}(a)$ and $\beta_{k}(a,b)$ as it is defined in Definitions \ref{alpha}, \ref{beta} and let $Fib_{k}$ be the $k$-th element of the Fibonacci sequence over $\mathbb{F}_{q}$. If $a=3$ and $b=1$ then for every $k\geq 0$ the following are satisfied:
\begin{itemize}
\item $\beta_{k}=Fib_{2k+1}$;
\item $\alpha_{k}=Fib_{2k+2}$.
\end{itemize}
\end{proposition}

\begin{proof}
Since for every $k\geq 0$, $Fib_{k+2}=Fib_{k+1}+Fib_{k}$, the following recurrence relation is satisfied:
$$Fib_{k+4}=2\cdot Fib_{k+2}+Fib_{k+1}=3\cdot Fib_{k+2}-(Fib_{k+2}-Fib_{k+1})=3\cdot Fib_{k+2}-Fib_{k}.$$
By Definition  \ref{alpha}, $\alpha_{0}(3)=1=Fib_{2}$, $\alpha_{1}(3)=3=Fib_{4}$ and since $b=1$, by Definition \ref{beta}, $\beta_{0}(3,1)=b=1=Fib_{1}$ and $\beta_{1}(3,1)=\alpha_{1}(3)-\alpha_{0}(3)=3-1=2=Fib_{3}$. Since by Proposition \ref{alpha-recursive}, $\alpha_{k+1}(3)=3\alpha_{k}(3)-\alpha_{k-1}(3)$, and by Proposition \ref{beta-gamma-linear}, $\beta_{k+1}(3,1)=3\cdot \beta_{k}(3,1)-\beta_{k-1}(3,1)$, we conclude the results of the proposition.
\end{proof}
\\

In this paper, we generalize Proposition \ref{bn-canon-psl} in the following way: Theorem \ref{main} demonstrates that every element of $g\in PSL_{2}(q)$ has a unique presentation of the form:

 $$g=[u(a)\cdot s]^{k}\cdot u(x)\cdot h(y),$$
where $0\leq k\leq q$, $\mathbb{F}_{q}$ has characteristic $2$, $x,y\in \mathbb{F}_{q}$ such that $y\neq 0$. The element $a\in \mathbb{F}_{q}$ is chosen such that the polynomial $\lambda^{2}-a\lambda+1$ is an irreducible polynomial over $\mathbb{F}_{q}$ and the root $\omega$ of the polynomial has order $q+1$ in $\mathbb{F}_{q^2}^{*}$;

$$g=[u(a)\cdot s]^{k}\cdot [u(b)\cdot s\cdot u(-b)]^{\ell}\cdot u(x)\cdot h(y),$$
where $0\leq k<\frac{q+1}{2}$, $0\leq \ell\leq 1$ (i. e. $l=0$ or $l=1$),  $\mathbb{F}_{q}$ does not have characteristic $2$, $x,y\in \mathbb{F}_{q}$ such that $y\neq 0$. The element $a\in \mathbb{F}_{q}$ is chosen such that the polynomial $\lambda^{2}-a\lambda+1$ is an irreducible polynomial over $\mathbb{F}_{q}$ and the root $\omega$ of the polynomial has order $q+1$ in  $\mathbb{F}_{q^2}^{*}$. The element $b\in \mathbb{F}_{q}$ is chosen in a way such that $u(b)\cdot s$ is not a left-coset representative of $[u(a)\cdot s]^{k}$ in $B$ for any $0\leq k\leq q$.

\begin{remark}
Since $U$ and $H$ are abelian groups, where $U$ is a direct sum of copies of $\mathbb{Z}_{p}$ such that $q=p^{\kappa}$ for some $\kappa$ (i.e., $p$ is the characteristic of $\mathbb{F}_{q}$), and $H$ is abelian as a quotient of the multiplicative group of $\mathbb{F}_{q}$, the elements $u(a)\cdot s$, ~$u(b)\cdot s\cdot u(-b)$, ~the $\kappa$ copies of $\mathbb{Z}_{p}$ (which generates the additive group of $\mathbb{F}_{q}$), and  an $OGS$ of $H$ ~ forms an $OGS$ for $PSL_{2}(q)$.
\end{remark}

Then, we look at the $BN-pair$ presentation of the elements $[u(a)\cdot s]^{k}$ and \\ $[u(a)\cdot s]^{\ell}\cdot [u(b)\cdot s\cdot u(-b)]$, $for 1\leq k\leq t-1$, and $0\leq \ell\leq t-1$, where $t=\frac{q+1}{gcd(2,q+1)}$:

$$[u(a)\cdot s]^{k}=u(a_k)\cdot s\cdot u(x_k)\cdot h(y_k)$$

$$[u(a)\cdot s]^{\ell}\cdot [u(b)\cdot s\cdot u(-b)]=u(b_\ell)\cdot s\cdot u({x'}_\ell)\cdot h({y'}_\ell).$$

We show that $a_{k}$, $x_{k}$, $y_{k}$, $b_{\ell}$, ${x'}_{\ell}$, and ${y'}_{\ell}$  are closely connected to the Dickson polynomial of the second kind over $\mathbb{F}_{q}$ (in case where $b=\frac{a}{2}$, by Corollary \ref{chebyshev}, the sequences $b_{\ell}$, ${x'}_{\ell}$, and ${y'}_{\ell}$ are closely connected to the Dickson and the Chebyshev polynomial of the first kind as well), and we connect it to the structure of the finite simple group $PSL_{2}(q)$.

\section{The connection of $OGS$ and $BN-pair$ presentation of $PSL_{2}(q)$ to Dickson polynomials}

In this section, we introduce recursive sequences over $\mathbb{F}_{q}$, which connects an $OGS$ presentation to the $BN-pair$ presentation of $PSL_{2}(q)$. We show some interesting properties of the sequences, and we show also, that the sequences are closely connected to Dickson and Chebyshev polynomials over $\mathbb{F}_{q}$.

\begin{definition}\label{pa}
For $a\in \mathbb{F}$, define $P_{a}(\lambda)$ to be the polynomial
$$P_{a}(\lambda)=\lambda^{2}-a\cdot \lambda+1.$$
\end{definition}

\begin{proposition}\label{pa-property}
Let $P_{a}(\lambda)$ be the polynomial as it is defined in Definition \ref{pa}, then the following holds:
\begin{itemize}
\item $P_{a}(\lambda)$ is the characteristic polynomial of the representative matrix for $u_{a}\cdot s$;
\item $P_{a}(0)=P_{a}(a)=1$;
\item $P_{a}(r)=P_{a}(a-r)$ for every $a, r\in \mathbb{F}_{q}$;
\item $P_{a}(\lambda)=r$ has a double root if and only if the following holds
\begin{itemize}
\item $q$ is odd;
\item The root $\lambda=\frac{a}{2}$;
\item $r=1-\frac{a^{2}}{4}$.
\end{itemize}
\item $P_{3}(1)=-1$.
\end{itemize}
\end{proposition}

\begin{proposition}\label{coset-order}
Let $a\in \mathbb{F}_{q}$, such that the polynomial $P_{a}(\lambda)$ is irreducible over $\mathbb{F}_{q}$, then the following holds:
\begin{itemize}
\item The root $\omega$ of $P_{a}(\lambda)$ has order which divides $q+1$ in the multiplicative group $\mathbb{F}_{q^2}^{*}$;
\item It is possible to choose $a\in \mathbb{F}_{q}$, such that the root $\omega$ of $P_{a}(\lambda)$ has order $q+1$ in the multiplicative group $\mathbb{F}_{q^2}^{*}$;
\item For  $a\in \mathbb{F}_{q}$, such that the root of $P_{a}(\lambda)$ has order $q+1$ in the multiplicative group $\mathbb{F}_{q^2}^{*}$, it is satisfied that the order of the element $u(a)\cdot s$ in $PSL_{2}(q)$ is $\frac{q+1}{gcd(2, q+1)}$.
 \end{itemize}
\end{proposition}

\begin{proof}
The first two parts of the proposition are obvious by basic theory of finite fields. Since by Proposition \ref{pa-property} $P_{a}(\lambda)$ is the characteristic polynomial of $u(a)\cdot s$, the third part is obvious by Cayley-Hamilton Theorem.
\end{proof}

\begin{theorem}\label{main}
Let $G=PSL_{2}(q)$, where $\mathbb{F}_{q}$ is a finite field of order $q$. Let $a\in \mathbb{F}_{q}$, such that the polynomial $P_{a}(\lambda)$ is irreducible over $\mathbb{F}_{q}$ and the root $\omega$ of the polynomial has order $q+1$ in  $\mathbb{F}_{q^2}^{*}$. In case of an odd $q$, let $b\in \mathbb{F}_{q}$, such that  $u(b)\cdot s$ is not a left-coset representative of $[u(a)\cdot s]^{k}$ in $B$ for any $0\leq k<\frac{q+1}{2}$.  Then every $g\in G$  has a unique presentation in the following form:

\begin{itemize}
\item If the field $\mathbb{F}_{q}$ has characteristic $2$: $$g=[u(a)\cdot s]^{k}\cdot u(x)\cdot h(y)$$ where $x\in \mathbb{F}_{q}$, $y\in \mathbb{F}_{q}^{*}$, and $0\leq k\leq q$;
\item If the field $\mathbb{F}_{q}$ has an odd characteristic: $$g=[u(a)\cdot s]^{k}\cdot [u(b)\cdot s\cdot u(-b)]^{\ell}\cdot u(x)\cdot h(y)$$ where $x\in \mathbb{F}_{q}$, $y\in \mathbb{F}_{q}^{*}$ ,$0\leq k<\frac{q+1}{2}$, $0\leq \ell\leq 1$ (i. e. $\ell=0$ or $\ell=1$).
\end{itemize}
\end{theorem}

\begin{proof}
First, we consider the case in which $\mathbb{F}_{q}$ has characteristic $2$. In that case, $PSL_{2}(q)=SL_{2}(q)$, and by Proposition \ref{coset-order}, there exists an element $a\in \mathbb{F}_{q}$, such that the order of the element $u(a)\cdot s$ is $q+1$. By Claim \ref{bn-matrix}, the  elements of the form $u(x)\cdot h(y)$ such that $x\in \mathbb{F}_{q}$ and $y\in \mathbb{F}_{q}^{*}$ form the subgroup $B$, which can be considered as the upper triangular matrices in $SL_{2}(q)$ (which is $PSL_{2}(q)$ in case of $q=2^m$). Therefore, $B=\{u(x)\cdot h(y) | x\in \mathbb{F}_{q}, y\in \mathbb{F}_{q}^{*}\}$  is a subgroup of order $q\cdot (q-1)$, where in case of $q=2^{m}$ that order is prime to $q+1$, the order of of $u(a)\cdot s$. Since the index of $B$ in $PSL_{2}(q)$ is $q+1$ too, the $q+1$ left-coset representatives of $B$ in $PSL_{2}(q)$ can be considered the elements which have the form of $[u(a)\cdot s]^{k}$, where $0\leq k\leq q$. Thus, every $g\in PSL_{2}(q)$ has a unique presentation of the form $$g=[u(a)\cdot s]^{k}\cdot u(x)\cdot h(y)$$ where $0\leq k\leq q$ and the $\mathbb{F}_{q}$ has characteristic $2$.
Now, we consider the case where $q$ is an odd prime. In this case, $PSL_{2}(q)$ is a proper quotient of $SL_{2}(q)$. Like in the case of $q=2^{m}$, the subgroup $B=\{u(x)\cdot h(y) | x\in \mathbb{F}_{q}, y\in \mathbb{F}_{q}^{*}\}$ has index $q+1$ in $PSL_{2}(q)$, but the order of $B$ is $\frac{q\cdot(q+1)}{2}$ (since $h(y)=h(-y)$ for every $y\in \mathbb{F}_{q}^{*}$), and by Proposition \ref{coset-order}, there exists an element $a\in \mathbb{F}_{q}$, such that the order of the element $u(a)\cdot s$ is just $\frac{q+1}{2}$. Thus all the elements of the form $$[u(a)\cdot s]^{k},$$ where $0\leq k<\frac{q+1}{2}$ gives just a half of the left-coset representatives of $B$ in $PSL_{2}(q)$, where $q$ is an odd prime. Now, we choose an element $b\in \mathbb{F}_{q}$, such that $$u(b)\cdot s\notin [u(a)\cdot s]^{k}B,$$
for any $0\leq k<\frac{q+1}{2}$. Then,
$$u(b)\cdot s\cdot u(-b)$$
is an involution conjugate to $s$, which does not belong to $$[u(a)\cdot s]^{k}B,$$ for any $0\leq k<\frac{q+1}{2}$ too. Now, we show that the $q+1$ elements of the form
$$[u(a)\cdot s]^{k}\cdot [u(b)\cdot s\cdot u(-b)]^{\ell},$$
where $0\leq k<\frac{q+1}{2}$ and $0\leq \ell\leq 1$, form a full left-coset representative of $B$ in $PSL_{2}(q)$ for an odd $q$. It is enough to show that all $q+1$ cosets of the form
$$[u(a)\cdot s]^{k}\cdot [u(b)\cdot s\cdot u(-b)]^{\ell}B$$
are different, where $0\leq k<\frac{q+1}{2}$ and $0\leq \ell\leq 1$.
Assume
$$[u(a)\cdot s]^{k_1}\cdot [u(b)\cdot s\cdot u(-b)]^{{\ell}_{1}}B=[u(a)\cdot s]^{k_2}\cdot [u(b)\cdot s\cdot u(-b)]^{{\ell}_{2}}B,$$
where $0\leq k_1, k_2<\frac{q+1}{2}$, and $0\leq {\ell}_{1},{\ell}_{2}\leq 1$. If ${\ell}_{1}={\ell}_{2}=0$, the equation implies that
 $$[u(a)\cdot s]^{k_1-k_2}\in B,$$
which is possible only if $k_1-k_2=0$, since the order of $[u(a)\cdot s]^{k_1-k_2}$ is otherwise prime to the order of $B$. If ${\ell}_{1}={\ell}_{2}=1$, the equation implies that
$$[u(b)\cdot s\cdot u(-b)]\cdot[u(a)\cdot s]^{k_1-k_2}\cdot [u(b)\cdot s\cdot u(-b)]\in B,$$ which is also possible only if $k_1-k_2=0$ by the same reason. If ${\ell}_{1}=1$ and ${\ell}_{2}=0$, then the equation implies that
$$[u(b)\cdot s\cdot u(-b)]B=[u(a)\cdot s]^{k_1-k_2}B,$$ which is impossible by the definition of $b$. Thus, all the $q+1$ cosets of the form
$$[u(a)\cdot s]^{k}\cdot [u(b)\cdot s\cdot u(-b)]^{\ell}B$$
are different.
\end{proof}
\\

The rest of the paper deals with the $OGS$ of $PSL_{2}(q)$, which has been showed in Theorem \ref{main}, and its connection to the $BN-pair$ decomposition and to Dickson and Chebyshev polynomials over $\mathbb{F}_{q}$. Hence, from now on, we use the following notation
\begin{itemize}
\item Denote by $a$ an element of $\mathbb{F}_{q}$, such that the root $\omega$ of the polynomial $P_{a}(\lambda)$ has order $q+1$ in  $\mathbb{F}_{q^2}^{*}$;
\item In case of an odd $q$, denote by $b$ an element of $\mathbb{F}_{q}$, such that $u(b)\cdot s$ is not a left-coset representative of $[u(a)\cdot s]^{k}$ in $B$ for any $0\leq k<\frac{q+1}{2}$;
\item Denote by $t$ the order of $u(a)\cdot s$ (i.e., $t=\frac{q+1}{gcd(2, q-1)}$).
\end{itemize}

\begin{proposition}\label{ab-unique}
Let $G=PSL_{2}(q)$. Consider the $BN-pair$ presentation of  $$[u(a)\cdot s]^{k},$$
for $1\leq k\leq t-1$, and in case of odd $q$, consider the $BN-pair$ presentation of $$[u(a)\cdot s]^{\ell}\cdot [u(b)\cdot s\cdot u(-b)]$$
too, for $0\leq \ell\leq t-1$.   If
$$[u(a)\cdot s]^{k}=u(a_k)\cdot s\cdot u(x_k)\cdot h(y_k),$$ for some $x_k\in \mathbb{F}_{q}$, $y_k\in \mathbb{F}_{q}^{*}$, and
$$[u(a)\cdot s]^{\ell}\cdot [u(b)\cdot s\cdot u(-b)]=u(b_k)\cdot s\cdot u({x'}_{\ell})\cdot h({y'}_{\ell}),$$ for some ${x'}_{\ell}\in \mathbb{F}_{q}$, ${y'}_{\ell}\in \mathbb{F}_{q}^{*}$

then

we get the following observations:
\begin{itemize}
\item In case $\mathbb{F}_{q}$ has characteristic $2$, the $q$ elements $a_{1}, a_{2}, \ldots, a_{q}$ are all the $q$ different elements of $\mathbb{F}_{q}$;
\item In case of $\mathbb{F}_{q}$ has an odd characteristic, the $q$ elements \\ $a_{1}, a_{2}, \ldots, a_{t-1}, b_{0}, b_{1}, \ldots, b_{t-1}$ are all the $q$ different elements of $\mathbb{F}_{q}$.
\end{itemize}
\end{proposition}

\begin{proof}
By Theorem \ref{main}, the $q$ elements of the form
$$[u(a)\cdot s]^{k}$$
are the $q$ different left-coset representatives of all the elements which are not in $B$, in case $\mathbb{F}_{q}$ has characteristic $2$. Thus, by applying Conclusion \ref{coset-bn-unique}, $a_{1}, a_{2}, \ldots, a_{q}$ are $q$ different elements of $\mathbb{F}_{q}$, and since $q$ is finite, obviously, these elements are all the $q$ different elements of $\mathbb{F}_{q}$. Similarly, in case of odd $q$, the elements
$$[u(a)\cdot s]^{k},$$ where $1\leq k<t$, and
$$[u(a)\cdot s]^{\ell}\cdot [u(b)\cdot s\cdot u(-b)],$$
where $0\leq \ell<t$ are the $q$ different left-coset representatives of all the elements which are not in $B$. Thus, by the same argument as in the case of even $q$, we get the desired result.
\end{proof}

Now, we present a recursive algorithm to find  $a_{k}$ and  $b_{\ell}$, and we also describe some interesting properties of these elements.

\begin{proposition}\label{t-1}
Let $G=PSL_{2}(q)$.  For $1\leq k\leq t-1$, let $a_{k}$ be elements of $\mathbb{F}_{q}$ as it is defined in Proposition \ref{ab-unique}. Then, $a_{t-1}=0$.
\end{proposition}

\begin{proof}
$$[u(a)\cdot s]^{t-1}=[u(a)\cdot s]^{-1}=s\cdot u(-a)=u(0)\cdot s\cdot u(-a)\cdot h(1).$$
Thus by its definition in Proposition \ref{ab-unique},
$$a_{t-1}=0.$$
\end{proof}

\begin{proposition}\label{ab-recursive}
Let $G=PSL_{2}(q)$. For every $1\leq k\leq t-1$ and $0\leq \ell\leq t-1$, let $a_{k}, b_{\ell}$ be elements of $\mathbb{F}_{q}$ as it is defined in Propositions \ref{ab-unique}. Then the following holds:
\begin{itemize}
\item  $a_{k+1}=a-a_{k}^{-1}$;
\item $b=b_{0}=a-b_{t-1}^{-1}$;
\item $b_{\ell+1}=a-b_{\ell}^{-1}$.
\end{itemize}
\end{proposition}

\begin{proof}
By Proposition \ref{ab-unique},
$$[u(a)\cdot s]^{k+1}=u(a_{k+1})\cdot s\cdot u(x_{k+1})\cdot h(y_{k+1}),$$ for some $x_{k+1}\in \mathbb{F}_{q}$, $y_{k+1}\in \mathbb{F}_{q}^{*}$.
On the other hand,
$$[u(a)\cdot s]^{k+1}=u(a)\cdot s\cdot [u(a)\cdot s]^{k}=u(a)\cdot s\cdot u(a_{k})\cdot s\cdot u(x_{k})\cdot h(y_{k}),$$
for some $x_k\in \mathbb{F}_{q}$, $y_k\in \mathbb{F}_{q}^{*}$. Now, by Proposition \ref{uhs-multiply},
$$s\cdot u\left(a_{k}\right)\cdot s=u\left(-a_{k}^{-1}\right)\cdot s\cdot u\left(-a_{k}\right)\cdot h\left(a_{k}\right).$$
 Therefore,
 $$[u\left(a\right)\cdot s]^{k+1}=u\left(a\right)\cdot u\left(-a_{k}^{-1}\right)\cdot s\cdot u\left(-a_{k}\right)\cdot h\left(a_{k}\right).$$
 Hence,
 $$[u\left(a\right)\cdot s]^{k+1}=u\left(a-a_{k}^{-1}\right)\cdot s\cdot u\left(-a_{k}\right)\cdot h\left(a_{k}\right),$$ for some $x_k\in \mathbb{F}_{q}$, $y_k\in \mathbb{F}_{q}^{*}$.
 Thus,
 $$a_{k+1}=a-a_{k}^{-1}.$$
 Similarly, by Proposition \ref{ab-unique},
 $$[u(a)\cdot s]^{l+1}\cdot [u(b)\cdot s\cdot u(-b)]=u(b_{\ell+1})\cdot s\cdot u({x'}_{\ell+1})\cdot h({y'}_{\ell+1}).$$
 On the other hand,
 $$[u(a)\cdot s]^{\ell+1}\cdot [u(b)\cdot s\cdot u(-b)]=u(a)\cdot s\cdot [u(a)\cdot s]^{\ell}\cdot [u(b)\cdot s\cdot u(-b)]=u(a)\cdot s\cdot u(b_{\ell})\cdot s\cdot u({x'}_{k})\cdot h({y'}_{k}).$$
 Now, by Proposition \ref{uhs-multiply},
 $$s\cdot u(b_{\ell})\cdot s=u(-b_{\ell}^{-1})\cdot s\cdot u(-b_{\ell})\cdot h(b_{\ell}).$$
  Hence,
  $$[u(a)\cdot s]^{\ell+1}\cdot [u(b)\cdot s\cdot u(-b)]=u(a-b_{\ell}^{-1})\cdot s\cdot u({x'}_{\ell+1})\cdot h({y'}_{\ell+1}).$$
  Thus,
  $$b_{\ell+1}=a-b_{\ell}^{-1}.$$
\end{proof}

The next theorem demonstrates interesting connections of the sequences $a_{k}$ and $b_{\ell}$ for $1\leq k\leq t-1$ and $0\leq \ell\leq t-1$, to $\alpha_{r}(a)$ (the Dickson polynomial of the second kind on variable $a$ over $\mathbb{F}_{q}$) and $\beta_{r}(a,b)$ (Where by Definition \ref{beta}, $\beta_{r}(a,b)$ is a linear combinations of $\alpha_{r}(a)$ and $\alpha_{r-1}(a)$, and in case $b=\frac{a}{2}$, by Proposition \ref{chebyshev}, $\beta_{r}(a,b)$ is the Chebyshev polynomial of the first kind on variable $b$ over $\mathbb{F}_{q}$).

\begin{theorem}\label{ab-dickson}
Let $G=PSL_{2}(q)$. For every $1\leq k\leq t-1$,  $0\leq \ell\leq t-1$, and $-1\leq r\leq t-1$, let $a_{k}$, $b_{\ell}$, $\alpha_{r}(a)$, and $\beta_{r}(a,b)$ be elements of $\mathbb{F}_{q}$ as it is defined in Proposition \ref{ab-unique} and in Definitions \ref{alpha}, \ref{beta}. Then the following holds:
\begin{itemize}
\item $a_{k}=\frac{\alpha_{k}(a)}{\alpha_{k-1}(a)}$;
\item $b_{\ell}=\frac{\beta_{\ell}(a,b)}{\beta_{\ell-1}(a,b)}$;
\item $\alpha_{k}(a)=\prod_{i=1}^{k}a_{i}$, where $1\leq k\leq t-1$;
\item $\beta_{\ell}(a,b)=\prod_{i=0}^{\ell}b_{i}$, where $0\leq \ell\leq t-1$.
\end{itemize}
\end{theorem}

\begin{proof}
The proof is by induction on $k$. By Definition \ref{beta},
$$\beta_{-1}(a,b)=1$$
  and
  $$\beta_{k}(a,b)=b\cdot\alpha_{k}(a)-\alpha_{k-1}(a),$$
  for $0\leq k\leq t-1$. Thus
  $$\beta_{0}(a,b)=b\cdot \alpha_{0}(a)-\alpha_{-1}(a)=b\cdot 1-0=b.$$
  Now, by Definition \ref{alpha},
  $$a_{1}=a=\alpha_{1}(a)\cdot\alpha_{0}^{-1}(a)$$
  and by Proposition \ref{ab-recursive},
  $$b_{0}=b=\beta_{0}(a,b)\cdot\beta_{-1}^{-1}(a,b)$$
  and
  $$b_{1}=a-b^{-1}=\frac{ab-1}{b}=\frac{b\cdot \alpha_{1}(a)-\alpha_{0}(a)}{b\cdot \alpha_{0}(a)-\alpha_{-1}(a)}=\frac{\beta_{1}(a,b)}{\beta_{0}^{-1}(a,b)}.$$
  Assume
  $$a_{\ell}=\frac{\alpha_{\ell}(a)}{\alpha_{\ell-1}^{-1}(a)}$$
  and
  $$b_{\ell}=\frac{b\cdot\alpha_{\ell}(a)-\alpha_{\ell-1}(a)}{b\cdot\alpha_{\ell-1}(a)-\alpha_{\ell-2}(a)}$$
  for every $\ell\leq k$. By Proposition \ref{ab-recursive},
  $$a_{k+1}=a-a_{k}^{-1}$$
  and
  $$b_{k+1}=a-b_{k}^{-1}.$$
  Then, by the induction hypothesis,
  $$a_{k+1}=a-\frac{\alpha_{k-1}(a)}{\alpha_{k}^{-1}(a)}$$
  and
  $$b_{k+1}=a-\frac{b\cdot\alpha_{k-1}(a)-\alpha_{k-2}(a)}{b\cdot\alpha_{k}(a)-\alpha_{k-1}(a)}.$$
  Thus,
  $$a_{k+1}=\frac{a\cdot\alpha_{k}(a)-\alpha_{k-1}(a)}{\alpha_{k}(a)}$$
  and
  $$b_{k+1}=\frac{b\cdot[a\cdot\alpha_{k}(a)-\alpha_{k-1}(a)]-[a\cdot\alpha_{k-1}(a)-\alpha_{k-2}(a)]}{b\cdot\alpha_{k}(a)-\alpha_{k-1}(a)}.$$ Therefore,
  $$a_{k+1}=\frac{\alpha_{k+1}(a)}{\alpha_{k}(a)}$$
  and
  $$b_{k+1}=\frac{b\cdot\alpha_{k+1}(a)-\alpha_{k}(a)}{b\cdot\alpha_{k}(a)-\alpha_{k-1}(a)}.$$
  Thus, the proposition holds for every $1\leq k\leq t-1$.
\end{proof}
\\

The next three propositions shows some interesting relations which comes from the definitions of  $a_{k}$, $b_{\ell}$, where $1\leq k\leq t-1$, $0\leq \ell\leq t-1$.

\begin{proposition}\label{prod-sum-a}
Let $G=PSL_{2}(q)$. For every $1\leq k\leq t-1$, let $a_{k}$ be elements of $\mathbb{F}_{q}$ as it is defined in Proposition \ref{ab-unique}. Then, the following holds:
\begin{itemize}
\item  $a_{k}\cdot a_{t-k-1}=1$ for $1\leq k\leq t-2$;
\item $a_{k}+a_{t-k}=a$.
\end{itemize}
\end{proposition}

\begin{proof}
The proof of the first part is by induction on $k$. First, we show the proposition for $k=1$. By Proposition \ref{ab-unique},
$$[u(a)\cdot s]^{t-2}=u(a_{t-2})\cdot s\cdot u(x_{t-2})\cdot h(y_{t-2}).$$
We also have
$$[u(a)\cdot s]^{t-2}=[u(a)\cdot s]^{-2}=s\cdot u(-a)\cdot s\cdot u(-a).$$
Then, by Proposition \ref{uhs-multiply},
$$s\cdot u(-a)\cdot s=u(a^{-1})\cdot s\cdot u(a)\cdot h(-a).$$
Therefore,
$$[u(a)\cdot s]^{t-2}=u(a^{-1})\cdot s\cdot u(a)\cdot h(a)\cdot u(-a).$$
Thus,
$$a_{t-2}=a^{-1}=a_{1}^{-1}.$$
Now, assume by induction
$$a_{t-r}=a_{r-1}^{-1}$$
for every $2\leq r\leq k$, and we prove the correctness of the proposition for $r=k+1$. By Proposition \ref{ab-unique},
$$[u(a)\cdot s]^{t-k-1}=u(a_{t-k-1})\cdot s\cdot u(x_{t-k-1})\cdot h(y_{t-k-1}).$$
We also have
\begin{align*}
[u(a)\cdot s]^{t-k-1}=[u(a)\cdot s]^{-(k+1)} &= s\cdot u(-a)\cdot [u(a)\cdot s]^{t-k} \\ &= s\cdot u(-a)\cdot u(a_{t-k})\cdot s\cdot u(x_{t-k})\cdot h(y_{t-k}).
\end{align*}
Since
$$a_{t-k}=a_{k-1}^{-1}$$
by our induction hypothesis, we have
$$[u(a)\cdot s]^{t-k-1}= s\cdot u(-a)\cdot u(a_{k-1}^{-1})\cdot s\cdot u(x_{t-k})\cdot h(y_{t-k}).$$
Then, by Proposition \ref{uhs-multiply},
$$s\cdot u(-a+a_{k-1}^{-1})\cdot s=u([a-a_{k-1}^{-1}]^{-1})\cdot s\cdot u(a-a_{k-1}^{-1})\cdot h(-a+a_{k-1}^{-1}).$$
Now, by using Proposition \ref{ab-recursive},
$$a-a_{k-1}^{-1}=a_{k}.$$
 Therefore,
 $$[u(a)\cdot s]^{t-k-1}= u({a_{k}}^{-1})\cdot s\cdot u(x_{t-k-1})\cdot h(y_{t-k-1}).$$
 Hence,
 $$a_{t-k-1}=a_{k}^{-1}$$
 for every $1\leq k\leq t-2$. \\
Now, we turn to the proof of the second part of the proposition. By Proposition \ref{t-1}, $a_{t-1}=0$. Thus,
$$a_1+a_{t-1}=a+0=a.$$
Now, for $2\leq k\leq t-1$,  we have
$$a_{k}=a-a_{k-1}^{-1}$$
by using Proposition \ref{ab-recursive}. Then by the first part of the proposition,
$$a_{k-1}^{-1}=a_{t-k}$$
for every $2\leq k\leq t-1$. Thus, we get
$$a_{k}+a_{t-k}=a$$
for every $1\leq k\leq t-1$.
\end{proof}

\begin{corollary}\label{a-middle}
Let $G=PSL_{2}(q)$. For every $1\leq k\leq t-1$, let $a_{k}$ be elements of $\mathbb{F}_{q}$ as it is defined in Proposition \ref{ab-unique}. Then, the following holds:
\begin{itemize}
\item In case the characteristic of $\mathbb{F}_{q}$ equals $2$: $$a_{\frac{q}{2}}=1;$$
\item In case the characteristic of $\mathbb{F}_{q}$ does not equal $2$, and $4 | q-1$: $$a_{\frac{q-1}{4}}=1$$ or $$a_{\frac{q-1}{4}}=-1;$$
\item In case the characteristic of $\mathbb{F}_{q}$ does not equal $2$, and $4 | q+1$: $$a_{\frac{q+1}{4}}=\frac{a}{2}.$$
\end{itemize}
\end{corollary}

\begin{proof}
First, assume the characteristic of $\mathbb{F}_{q}$ equals $2$. Then, $t=q+1$. Thus, for $k=\frac{q}{2}$ we have
$$t-k-1=q+1-\frac{q}{2}-1=\frac{q}{2}.$$
Therefore, by Proposition \ref{prod-sum-a},
$$(a_{\frac{q}{2}})^{2}=1.$$
Hence,
$$a_{\frac{q}{2}}=1$$
in case the characteristic of $\mathbb{F}_{q}$ equals $2$. \\
Now, assume the characteristic of $\mathbb{F}_{q}$ does not equal $2$. Then, $q$ is odd, and either $4 | q-1$ or $4 | q+1$. First, assume $4 | q-1$. In this case, $t=\frac{q+1}{2}$, and for $k=\frac{q-1}{4}$ we have
$$t-k-1=\frac{q+1}{2}-\frac{q-1}{4}-1=\frac{q-1}{4}.$$
Therefore, by Proposition \ref{prod-sum-a},
$$(a_{\frac{q-1}{4}})^{2}=1.$$
Hence, either
$$a_{\frac{q-1}{4}}=1$$
or
$$a_{\frac{q-1}{4}}=-1$$
in case $q$ is odd, and $4 | q-1$. \\
Now, assume $q$ is odd, and $4 | q+1$. In this case, $t=\frac{q+1}{2}$, and for $k=\frac{q+1}{4}$ we have
$$t-k=\frac{q+1}{2}-\frac{q+1}{4}=\frac{q+1}{4}.$$
Therefore, by Proposition \ref{prod-sum-a}
$$2\cdot a_{\frac{q+1}{4}}=a.$$
Hence,
$$a_{\frac{q+1}{4}}=\frac{a}{2}$$
in case $q$ is odd and $4 | q+1$.
\end{proof}

\begin{proposition}\label{akl}
Let $G=PSL_{2}(q)$. Then, for every $1\leq \ell<k\leq t-1$, the following holds:
$$a_{k}=a_{\ell}-[(a_{1}\cdot a_{2}\cdots a_{\ell-1})\cdot (a_{k-\ell}\cdots a_{k-3}\cdot a_{k-2})\cdot a_{k-1}]^{-1}.$$
In particular:
$$a_{k}=a_{k-1}-[a_{1}^{2}\cdot a_{2}^{2}\cdots a_{k-2}^{2}\cdot a_{k-1}]^{-1}.$$
\end{proposition}

\begin{proof}
The proof is by induction on $\ell$. For $\ell=1$, we get
$$a_{k}=a_{1}-{a_{k-1}}^{-1},$$
which holds by Proposition \ref{ab-recursive}. Now, we prove the proposition for $\ell=2$. First, by Proposition \ref{ab-recursive},
$$a=a_{2}+a_{1}^{-1}.$$
Thus, by using again Proposition \ref{ab-recursive}, we have
$$a_{k}=a-a_{k-1}^{-1}=a_{2}+a_{1}^{-1}-a_{k-1}^{-1}=a_{2}-(a_{1}-a_{k-1})\cdot a_{1}^{-1}\cdot a_{k-1}^{-1}.$$
Since by Proposition \ref{ab-recursive},
$$a_{1}-a_{k-1}=a_{k-2}^{-1},$$
we finally get
$$a_{k}=a_{2}-a_{1}^{-1}\cdot a_{k-2}^{-1}\cdot a_{k-1}^{-1}.$$
Hence, the Proposition holds for $l=2$. Now, assume by induction, for every $r\leq \ell-1<k$:
$$a_{k}=a_{r}-[(a_{1}\cdot a_{2}\cdots a_{r-1})\cdot (a_{k-r}\cdots a_{k-3}\cdot a_{k-2})\cdot a_{k-1}]^{-1}.$$
and we prove it for $r=\ell$. Similarly to the case of $\ell=2$, by Proposition \ref{ab-recursive}, we get
$$a=a_{\ell}+a_{\ell-1}^{-1}.$$
Thus,
$$a_{k}=a-a_{k-1}^{-1}=a_{\ell}+a_{\ell-1}^{-1}-a_{k-1}^{-1}=a_{\ell}-[a_{\ell-1}-a_{k-1}]\cdot[a_{\ell-1}\cdot a_{k-1}]^{-1}$$
by using again Proposition \ref{ab-recursive}. Since
$$a_{\ell-1}-a_{k-1}=[(a_{1}\cdot a_{2}\cdots a_{\ell-2})\cdot (a_{k-\ell}\cdots a_{k-4}\cdot a_{k-3})\cdot a_{k-2}]^{-1}$$
by using our induction hypothesis, we finally get
$$a_{k}=a_{\ell}-[(a_{1}\cdot a_{2}\cdots a_{\ell-1})\cdot (a_{k-\ell}\cdots a_{k-3}\cdot a_{k-2})\cdot a_{k-1}]^{-1}.$$
By substitution of $\ell=k-1$, which is the largest possible $\ell<k$, we derive the following result:
$$a_{k}=a_{k-1}-[a_{1}^{2}\cdot a_{2}^{2}\cdots a_{k-2}^{2}\cdot a_{k-1}]^{-1}.$$
\end{proof}

\begin{proposition}\label{b-property}
Let $G=PSL_{2}(q)$ such that $q$ is odd. For every $1\leq k\leq t-1$  $0\leq \ell\leq t-1$, let $a_{k}$ and $b_{\ell}$ be elements of $\mathbb{F}_{q}$ as it is defined in Proposition \ref{ab-unique}. Then, the following holds:
\begin{itemize}
\item $b_{0}=b$;
\item $b_{\ell}=\frac{a_{\ell}\cdot b-1}{b-a_{\ell-1}^{-1}}=\frac{a_{\ell}\cdot b-1}{b-a_{t-\ell}}=\frac{1-a_{\ell}\cdot b}{a-b-a_{\ell}}$, where $1\leq \ell\leq t-1$;
\item $b_{t-1}=(a-b)^{-1}$;
\item $\frac{b_{\ell}\cdot b_{\ell+1}}{a_{\ell}\cdot a_{\ell-1}}=\frac{a_{\ell+1}\cdot b-1}{a_{\ell-1}\cdot b-1}$, where $2\leq \ell\leq t-2$.
\end{itemize}
\end{proposition}

\begin{proof}
First, $b_{0}=b$ by definition. By Theorem \ref{ab-dickson}, $$b_{\ell}=\frac{b\cdot\alpha_{\ell}(a)-\alpha_{\ell-1}(a)}{b\cdot\alpha_{\ell-1}(a)-\alpha_{\ell-2}(a)}$$
and $$\alpha_{\ell}(a)=\prod_{i=1}^{\ell}a_{i},$$
where $1\leq \ell\leq t-1$. Thus,

\begin{align*}
b_{\ell}&=\frac{\prod_{i=1}^{\ell-1}a_{i}\cdot(a_{\ell}\cdot b-1)}{\prod_{i=1}^{\ell-2}a_{i}\cdot(a_{\ell-1}\cdot b-1)} \\ &=\frac{a_{\ell-1}\cdot(a_{\ell}\cdot b-1)}{a_{\ell-1}\cdot b-1} \\ &=\frac{a_{\ell}\cdot b-1}{b-{a_{\ell-1}}^{-1}}.
\end{align*}

 Then, by Proposition \ref{prod-sum-a},
 $${a_{\ell-1}}^{-1}=a_{t-\ell}=a-a_{\ell}.$$ Therefore,
 $$b_{\ell}=\frac{a_{\ell}\cdot b-1}{b-a_{t-\ell}}=\frac{1-a_{\ell}\cdot b}{a-b-a_{\ell}}.$$ Now, we show
 $$b_{t-1}=(a-b)^{-1}.$$
 By Proposition \ref{t-1}, $a_{t-1}=0$. Therefore, by substituting $\ell=t-1$ we get the desired result. Now, by using $$b_{\ell}=\frac{a_{\ell-1}\cdot(a_{\ell}\cdot b-1)}{a_{\ell-1}\cdot b-1},$$
 we get
 $$\frac{b_{\ell}\cdot b_{\ell+1}}{a_{\ell}\cdot a_{\ell-1}}=\frac{a_{\ell+1}\cdot b-1}{a_{\ell-1}\cdot b-1},$$
 where $2\leq \ell\leq t-2$.
\end{proof}

\begin{proposition}\label{b-1-property}
Let $G=PSL_{2}(q)$ such that $q$ is odd. For every $1\leq k\leq t-1$, $0\leq \ell\leq t-1$, let $a_{k}$ and $b_{\ell}$ be elements of $\mathbb{F}_{q}$ as it is defined in Proposition \ref{ab-unique}. Then the following holds:
\begin{itemize}
\item It is possible to choose either $b=1$ or $b=-1$;
\item If $b=1$ or $b=-1$, then for all $0\leq \ell\leq t-1$:
\begin{itemize}
\item $b_{\ell}\cdot b_{t-\ell}=1$, where $b_{t}=b_{0}=b$;
\item $b_{\ell}+b_{t-\ell+1}=a$.
\end{itemize}
\end{itemize}
\end{proposition}

\begin{proof}
First, consider the case $4 | q+1$ (i.e., $q\equiv 3 ~~mod ~~4$). Then, $t$ is even and, therefore, $t-k-1\neq k$ for any integer $k$, such that $1\leq k\leq t-1$. Therefore, by Proposition \ref{ab-unique}, $$a_{t-k-1}\neq a_{k},$$ where $1\leq k\leq t-1$. By Proposition \ref{prod-sum-a}, $$a_{t-k-1}=a_{k}^{-1},$$ therefore, $a_{k}\neq \pm{1}$, for any integer $k$ such that $1\leq k\leq t-1$. Thus, by Proposition \ref{ab-unique}, we may choose $b=1$ or $b=-1$ in case $4 | q+1$. Now, consider the case $4 | q-1$ (i.e., $q\equiv 1 ~~mod ~~4$). Then, $t$ is odd and, therefore, $t-k-1=k$ if and only if $k=\frac{t-1}{2}=\frac{q-1}{4}$, where $k$ is an integer such that $1\leq k\leq t-1$. Therefore, by the same argument as in the case of  $q\equiv 3 ~~mod ~~4$, $a_{k}\neq \pm{1}$, for any integer $k\neq \frac{q-1}{4}$ such that $1\leq k\leq t-1$. By Proposition \ref{a-middle}, $a_{\frac{q-1}{4}}=1$ or $a_{\frac{q-1}{4}}=-1$. Thus by Proposition \ref{ab-unique}, we may choose $b=1$ or $b=-1$, such that $b\neq a_{\frac{q-1}{4}}$, in case $4 | q-1$ too. Thus, we may assume either $b=1$ or $b=-1$. We prove the second part of the proposition. For $\ell=1$, we have by Proposition \ref{ab-recursive} $$b_{1}=a-b^{-1},$$
and for $\ell=t-1$, we have by Proposition \ref{b-property} $$b_{t-1}=(a-b)^{-1}.$$
Thus, in case of $b=1$ or $b=-1$, $$b_{1}\cdot b_{t-1}=\frac{a-b^{-1}}{a-b}=\frac{a-b}{a-b}=1.$$
By Proposition \ref{b-property}, $$b_{\ell}=\frac{1-a_{\ell}\cdot b}{a-b-a_{\ell}},$$
and $$b_{t-\ell}=\frac{1-a_{t-\ell}\cdot b}{a-b-a_{t-\ell}}.$$
By Proposition \ref{prod-sum-a}, $$a_{t-\ell}=a-a_{\ell},$$
Thus, $$b_{t-\ell}=\frac{1-(a-a_{\ell})\cdot b}{a_{\ell}-b}.$$
Thus, using $b=1$ or $b=-1$:
$$b_{\ell}\cdot b_{t-\ell}=\frac{1-a_{\ell}}{a-1-a_{\ell}}\cdot\frac{1-a+a_{\ell}}{a_{\ell}-1}=1,$$
or
$$b_{\ell}\cdot b_{t-\ell}=\frac{1+a_{\ell}}{a+1-a_{\ell}}\cdot\frac{1+a-a_{\ell}}{a_{\ell}+1}=1.$$
Now, by Proposition \ref{ab-recursive}, $$b_{t-\ell+1}=a-b_{t-\ell}^{-1}.$$
Thus, $$b_{\ell}+b_{t-\ell+1}=a.$$
\end{proof}

\begin{proposition}\label{b-h-property}
Let $G=PSL_{2}(q)$ such that $4 | q-1$. For every $1\leq k\leq t-1$, $0\leq \ell\leq t-1$, let $a_{k}$ and $b_{\ell}$ be elements of $\mathbb{F}_{q}$ as it is defined in Proposition \ref{ab-unique}. Then the following holds:
\begin{itemize}
\item It is possible to choose $b=\frac{a}{2}$;
\item If $b=\frac{a}{2}$, then for all $0\leq \ell\leq t-1$:
\begin{itemize}
\item $b_{\ell}+b_{t-\ell}=a$, where $b_{t}=b_{0}=b$;
\item $b_{\ell}\cdot b_{t-\ell-1}=1$;
\item $b_{\frac{t-1}{2}}=1$ or $b_{\frac{t-1}{2}}=-1$.
\end{itemize}
\end{itemize}
\end{proposition}

\begin{proof}
Assume $a_{k}=\frac{a}{2}$, for some $1\leq k\leq t-1$. Since by Proposition \ref{prod-sum-a}, $a_{t-k}=a-a_{k}$ for every $1\leq k\leq t-1$, we have that if $a_{k}=\frac{a}{2}$ then $a_{t-k}=a-\frac{a}{2}=\frac{a}{2}$ as well. Therefore, by Proposition \ref{ab-unique}, $k=t-k$, which means $t$ is even. Hence, in case $4 | q-1$ (i.e., the case where $t$ is odd), there is no $k$ such that $a_{k}=\frac{a}{2}$. Thus in that case, there is possible to choose $b=\frac{a}{2}$. Now, we prove the second part of the proposition.
For $\ell=t-1$, we have by Proposition \ref{b-property}
First, notice $b_{0}=\frac{a}{2}$. Hence, $$b_{t}+b_{0}=\frac{a}{2}+\frac{a}{2}=a.$$. By Proposition \ref{b-property}, $$b_{\ell}=\frac{1-a_{\ell}\cdot b}{a-b-a_{\ell}},$$
and $$b_{t-\ell}=\frac{1-a_{t-\ell}\cdot b}{a-b-a_{t-\ell}}.$$
Where, by using $b=\frac{a}{2}$, we have
$$b_{\ell}=\frac{2-a_{\ell}\cdot a}{a-2\cdot a_{\ell}}, ~~~~ b_{t-\ell}=\frac{a^{2}-a_{\ell}\cdot a-2}{a-2\cdot a_{\ell}}.$$
Hence, $$b_{\ell}+b_{t-\ell}=\frac{a^{2}-2\cdot a_{\ell}\cdot a}{a-2\cdot a_{\ell}}=a.$$
Now we turn to the proof of $b_{\ell}\cdot b_{t-\ell-1}=1$. First notice,
$$b_{t-1}=\left(a-b\right)^{-1}=\left(a-\frac{a}{2}\right)^{-1}=\left(\frac{a}{2}\right)^{-1}.$$
Hence, $$b_{0}\cdot b_{t-1}=1.$$
By Proposition \ref{ab-recursive}, $$b_{\ell+1}=a-b_{\ell}^{-1},$$ for every $0\leq \ell\leq t-1$. Hence, by using the argument $b_{\ell}+b_{t-\ell}=a$, we conclude
$$b_{\ell}^{-1}=a-b_{\ell+1}=b_{t-\ell-1}.$$
Then in case $\ell=\frac{t-1}{2}$: $$b_{\frac{t-1}{2}}^{2}=1.$$
Hence, $b_{\frac{t-1}{2}}=1$ or $b_{\frac{t-1}{2}}=-1$.
\end{proof}
\\

The next proposition shows some interesting properties of $\alpha_{k}(a)$, the Dickson polynomial of the second kind over $\mathbb{F}_{q}$ in variable $a$, and the related polynomials $\beta_{k}(a,b)$ and $\gamma_{k}(a,b)$, as it is defined in Definition \ref{beta}, which we conclude by using the theorems and the propositions in this section.

\begin{proposition}\label{alpha-property}
Let $G=PSL_{2}(q)$. For every $-1\leq r\leq t-1$, let  $\alpha_{r}(a)$, $\beta_{r}(a,b)$, $\gamma_{r}(a,b)$ be elements of $\mathbb{F}_{q}$ as it is defined in Definitions \ref{alpha}, \ref{beta}. Then the following holds:
\begin{itemize}
\item $\alpha_{t-1}(a)=0$;
\item $\alpha_{t-2}(a)=1$ or $\alpha_{t-2}(a)=-1$;
\item $\alpha_{t-k-2}(a)=\alpha_{k}$ or $\alpha_{t-k-2}(a)=-\alpha_{k}(a)$, where $-1\leq k\leq t-1$;
\item $\alpha_{\ell}(a)\cdot\alpha_{k-1}(a)-\alpha_{\ell-1}(a)\cdot\alpha_{k}(a)=\alpha_{k-\ell-1}(a)$, where $0\leq \ell\leq k\leq t-1$;
\item $\alpha_{k}^{2}(a)-\alpha_{k+1}(a)\cdot\alpha_{k-1}(a)=1$, where $0\leq k\leq t-2$;
\item $P_{a}(a_{k})=\alpha_{k-1}^{-2}(a)$, where $1\leq k\leq t-1$;
\item $\alpha_{k}(a)\cdot \beta_{k-1}(a,b)-\alpha_{k-1}(a)\cdot \beta_{k}(a,b)=1$, where $0\leq k\leq t-1$;
\item $\gamma_{k}(a,b)\cdot \beta_{k-1}(a,b)-\gamma_{k-1}(a,b)\cdot \beta_{k}(a,b)=1$, where $0\leq k\leq t-1$.
\item If $b=\frac{a}{2}$, then for every $0\leq k\leq t-1$ the following holds:
\begin{itemize}
\item $\beta_{t-1}(a,b)=1$ or $\beta_{t-1}(a,b)=-1$ (which means the Chebyshev polynomial of the first kind on variable $b$ satisfies $T_{t}(b)=1$ or $T_{t}(b)=-1$ over $\mathbb{F}_{q}$ for every $q\equiv 1 ~~mod ~~4$);
\item $\beta_{k}(a,b)=\beta_{t-1}(a,b)\cdot \beta_{t-k-2}(a,b)$;
\item $\beta_{k}^{2}(a,b)-\beta_{k+1}(a,b)\cdot\beta_{k-1}(a,b)=1-b^{2}$;
\item $P_{a}(b_{k})=\left(1-b^{2}\right)\cdot\beta_{k-1}^{-2}(a,b)$.
\end{itemize}
\item If $b=1$ or $b=-1$, then $\beta_{t-k-1}(a,b)=b\cdot \beta_{k}(a,b)$, where $0\leq k\leq t-1$.
\end{itemize}
\end{proposition}

\begin{proof}
By Proposition \ref{t-1}, $a_{t-1}=0$, and by Theorem \ref{ab-dickson}, $$a_{t-1}=\frac{\alpha_{t-1}(a)}{\alpha_{t-2}(a)}.$$
Thus, $\alpha_{t-1}(a)=0$. First, we prove
$$\alpha_{t-2}(a)\cdot \alpha_{k}(a)=\alpha_{t-k-2}(a),$$
for every $-1\leq k\leq t-1$. Notice that the statement holds for $k=-1$ and for $k=0$, since $\alpha_{-1}(a)=\alpha_{t-1}(a)=0$. Notice,
$$\alpha_{t-1}(a)=a\cdot \alpha_{t-2}(a)-\alpha_{t-3}(a).$$
Thus,
$$\alpha_{t-3}(a)=a\cdot \alpha_{t-2}(a).$$
Therefore, the statement holds for $k=1$ as well. \\ Now, assume by induction,
$$\alpha_{t-2}(a)\cdot \alpha_{r}(a)=\alpha_{t-r-2}(a),$$
for every $r\leq k$, and we prove the statement for $r=k+1$. By Theorem \ref{ab-dickson}, $$\alpha_{t-2}(a)\cdot\alpha_{k+1}(a)=\alpha_{t-2}(a)\cdot\alpha_{k}(a)\cdot a_{k+1}.$$
Then, by our induction hypothesis, we get
$$\alpha_{t-2}(a)\cdot\alpha_{k}(a)\cdot a_{k+1}=\alpha_{t-k-2}(a)\cdot a_{k+1}.$$
Then, by again using Theorem \ref{ab-dickson},
$$\alpha_{t-k-2}(a)\cdot a_{k+1}=\alpha_{t-k-3}(a)\cdot a_{t-k-2}\cdot a_{k+1}.$$ Hence, by Proposition \ref{prod-sum-a},
$$a_{t-k-2}\cdot a_{k+1}=1.$$ Therefore, we get
$$\alpha_{t-2}(a)\cdot\alpha_{k+1}(a)=\alpha_{t-k-3}(a).$$ Thus, the statement
$$\alpha_{t-2}(a)\cdot \alpha_{k}(a)=\alpha_{t-k-2}(a)$$
holds for every $-1\leq k\leq t-1$. Now, by substituting $k=t-2$, we get $$\alpha_{t-2}^{2}(a)=\alpha_{0}(a)=1.$$ Hence, either $$\alpha_{t-2}(a)=1$$
or $$\alpha_{t-2}(a)=-1$$
in case $q$ is odd. By the same argument, $$\alpha_{t-2}(a)=1,$$ in case $\mathbb{F}_{q}$ has characteristic $2$.
By using
$$\alpha_{t-2}(a)\cdot \alpha_{k}(a)=\alpha_{t-k-2}(a),$$
we have that either
$$\alpha_{t-k-2}(a)=\alpha_{k}(a)$$
 or
 $$\alpha_{t-k-2}(a)=-\alpha_{k}(a),$$
 in case $q$ is odd and
 $$\alpha_{t-k-2}(a)=\alpha_{k}(a)$$
 necessarily, in case $\mathbb{F}_{q}$ has characteristic $2$.
 Now, we turn to the next two parts of the proposition. If $k=\ell$, we have
$$\alpha_{k}(a)\cdot\alpha_{k-1}(a)-\alpha_{k-1}(a)\cdot\alpha_{k}(a)=\alpha_{-1}(a),$$
which leads to $0=0$, since $\alpha_{-1}=0$. Now, assume $k>\ell$. Then, by Proposition \ref{akl},
$$a_{\ell}-a_{k}=[(a_{1}\cdot a_{2}\cdots a_{\ell-1})\cdot (a_{k-\ell}\cdots a_{k-3}\cdot a_{k-2}\cdot a_{k-1})]^{-1}.$$
Now, by using Theorem \ref{ab-dickson}, $$\frac{\alpha_{\ell}(a)}{\alpha_{\ell-1}(a)}-\frac{\alpha_{k}(a)}{\alpha_{k-1}(a)}=\frac{\alpha_{k-\ell-1}(a)}{\alpha_{\ell-1}(a)\cdot \alpha_{k-1}(a)}.$$
Thus,
$$\alpha_{\ell}(a)\cdot\alpha_{k-1}(a)-\alpha_{\ell-1}(a)\cdot\alpha_{k}(a)=\alpha_{k-\ell-1}(a).$$
By considering the case $k=\ell+1$, and then substituting $k$ instead of $\ell$, we have
$$[\alpha_{k}(a)]^{2}-\alpha_{k+1}(a)\cdot \alpha_{k-1}(a)=\alpha_{0}(a)=1.$$ Thus,
\begin{align*}
(\alpha_{k}(a)+1)\cdot(\alpha_{k}(a)-1) &= \alpha^{2}_{k}(a)-1 \\ &= \alpha_{k+1}(a)\cdot \alpha_{k-1}(a).
\end{align*}
By substituting $\alpha_{k}(a)$ and $\alpha_{k-1}(a)$ for  $\alpha_{k+1}(a)$  according to Proposition \ref{alpha-recursive} we get
$$\alpha_{k}^{2}(a)-a\cdot \alpha_{k}(a)\cdot \alpha_{k-1}(a)+\alpha_{k-1}^{2}(a)=1,$$
which is by Theorem \ref{ab-dickson}, equivalent to
$$a_{k}^{2}-a\cdot a_{k}+1=P_{a}(a_{k})=\alpha_{k-1}^{-2}(a).$$
Now, by Definition \ref{beta},
$$\beta_{k}(a,b)=b\cdot\alpha_{k}(a)-\alpha_{k-1}(a).$$
Thus,
 \begin{align*}
1+\alpha_{k-1}(a)\cdot \beta_{k}(a,b) &= 1+\alpha_{k-1}(a)\cdot [b\cdot\alpha_{k}(a)-\alpha_{k-1}(a)] \\ &= 1+b\cdot\alpha_{k-1}(a)\cdot\alpha_{k}(a)-\alpha_{k-1}^{2}(a).
 \end{align*}
 Now, by using
 $$\alpha_{k-1}^{2}(a)-1=\alpha_{k}(a)\cdot\alpha_{k-2}(a),$$
 we get
 \begin{align*}
 1+\alpha_{k-1}(a)\cdot \beta_{k}(a,b) &= \alpha_{k}(a)\cdot[b\cdot\alpha_{k-1}(a)-\alpha_{k-2}(a)] \\ &=\alpha_{k}(a)\cdot\beta_{k-1}(a,b).
 \end{align*}
  The last part of the proposition is a consequence of it since
  \begin{align*}
  1+\gamma_{k-1}(a,b)\cdot \beta_{k}(a,b) &= 1+[\alpha_{k-1}(a)+b\cdot\beta_{k-1}(a,b)]\cdot \beta_{k}(a,b) \\&= [\alpha_{k}(a)+b\cdot\beta_{k}(a,b)]\cdot \beta_{k-1}(a,b) \\ &= \gamma_{k}(a,b)\cdot\beta_{k-1}(a,b).
  \end{align*}
Now, assume $b=\frac{a}{2}$. By Proposition \ref{b-h-property}, $b_{\frac{t-1}{2}}=1$ or $b_{\frac{t-1}{2}}=-1$. Hence, by Theorem \ref{ab-dickson}, $$\frac{\beta_{\frac{t+1}{2}}(a,b)}{\beta_{\frac{t-1}{2}}(a,b)}=b_{\frac{t-1}{2}},$$ which equals either to $1$ or to $-1$. Hence, either $$\beta_{\frac{t+1}{2}}(a,b)=\beta_{\frac{t-1}{2}}(a,b),$$ or $$\beta_{\frac{t+1}{2}}(a,b)=-\beta_{\frac{t-1}{2}}(a,b).$$ Now, since by Proposition \ref{b-h-property}, we have that $$b_{t-\ell-1}=b_{\ell}^{-1},$$
 by using Theorem \ref{ab-dickson}, we conclude that either $$\beta_{t-\ell-2}(a,b)=b_{\ell}(a,b)$$ or $$\beta_{t-\ell-2}(a,b)=-b_{\ell}(a,b),$$
and in case of $\ell=t-1$, we have either $$\beta_{t-1}(a,b)=1$$ or $$\beta_{t-1}(a,b)=-1.$$
Since by Proposition \ref{chebyshev} $\beta_{k}(a,b)$ is the Chebyshev polynomial of the first kind, $T_{k+1}(b)$ in variable $b$ over $\mathbb{F}_{q}$, $\beta_{k}(a,b)$ satisfies the following properties of Chebyshev polynomials of the first kind:
$$\beta_{k}^{2}(a,b)-\beta_{k+1}(a,b)\cdot\beta_{k-1}(a,b)=1-b^{2},$$ where $0\leq k\leq t-2$.
Now, since by Theorem \ref{ab-dickson}, $b_{k}=\frac{\beta_{k}(a,b)}{\beta_{k-1}(a,b)}$, and by Remark \ref{beta-gamma-linear}, $\beta_{k+1}=a\cdot \beta_{k}-\beta_{k-1}$, we conclude $$P_{a}(b_{k})=\left(1-b^{2}\right)\cdot\beta_{k-1}^{-2}(a,b),$$ where $1\leq k\leq t-1$.
Now, we turn to the case of $b=1$ or $b=-1$. By using Proposition \ref{b-1-property}, we get
 $$\beta_{t-k-1}(a,b)=b\cdot \beta_{k}(a,b),$$
 where $0\leq k\leq t-1$.

 \end{proof}

\begin{remark}
By Theorem \ref{main}, the group $PSL_{2}(q)$, for an odd $q$, has an $OGS$ presentation of the form:
$$[u(a)\cdot s]^{k}\cdot [u(b)\cdot s\cdot u(-b)]^{\ell}\cdot g,$$
such that:
\begin{itemize}
\item $0\leq k\leq t-1$;
\item $\ell=0$ or $\ell=1$;
\item $g$ is an element of the soluble subgroup $B$ of $PSL_{2}(q)$.
\end{itemize}
By Theorem \ref{b-a-a-b}, we will show the $OGS$ presentation of
$$[u(b)\cdot s\cdot u(-b)]\cdot [u(a)\cdot s]^{k}.$$
We will show some special interesting cases of it as well, which depends on some properties of the field $\mathbb{F}_{q}$.
\end{remark}

\begin{theorem}\label{b-a-a-b}
Let $q=p^{n}$, where $p$ is an odd prime, and $n\in \mathbb{N}$. For every $1\leq k\leq t-1$, $0\leq \ell\leq t-1$, and $-1\leq r\leq t-1$, let $a_{k}$, $b_{\ell}$, $\alpha_{r}(a)$, $\beta_{r}(a,b)$, and $\gamma_{r}(a,b)$ be elements of $\mathbb{F}_{q}$ as it is defined in Proposition \ref{ab-unique} and in Definitions \ref{alpha}, \ref{beta}.
Then for every $k_{1}$ and $k_{2}$ such that $0\leq k_{1}, k_{2}\leq t-1$ the following holds:
$$[u(b)\cdot s\cdot u(-b)]\cdot [u(a)\cdot s]^{k_{1}}\in [u(a)\cdot s]^{k_{2}}\cdot [u(b)\cdot s\cdot u(-b)]B$$
if and only if
$$\left(b-a_{t-k_{2}}\right)\cdot \left(a_{k_{1}}-b\right)^{-1}=\left(b-a+a_{k_{2}}\right)\cdot \left(a_{k_{1}}-b\right)^{-1}=P_{a}(b).$$
In particular, if $b$ satisfies
$$P_{a}(b)=-1 ~~~~~ (i.e., ~~ b^{2}-a\cdot b+2=0),$$
then for every $0\leq k\leq t-1$ we have:
$$[u(b)\cdot s\cdot u(-b)]\cdot [u(a)\cdot s]^{k}\in [u(a)\cdot s]^{-k}\cdot [u(b)\cdot s\cdot u(-b)]B$$

\end{theorem}

\begin{proof}
First, consider $[u(b)\cdot s\cdot u(-b)]\cdot [u(a)\cdot s]^{k_{1}}$ and $[u(a)\cdot s]^{k_{2}}\cdot [u(b)\cdot s\cdot u(-b)]$. By Proposition \ref{ab-unique},the following holds:
\begin{itemize}
\item $[u(a)\cdot s]^{k_{1}}=u(a_{k_{1}})\cdot s\cdot g$, such that $g\in B$;
\item $[u(a)\cdot s]^{k_{2}}\cdot [u(b)\cdot s\cdot u(-b)]=u(b_{k_{2}})\cdot s\cdot g'$, such that $g'\in B$.
\end{itemize}
Hence, by Proposition \ref{uhs-multiply},
\begin{align*}
[u(b)\cdot s\cdot u(-b)]\cdot [u(a)\cdot s]^{k_{1}}&=[u(b)\cdot s\cdot u(-b)]\cdot u(a_{k_{1}})\cdot s\cdot g \\ &=u(b-[a_{k_{1}}-b]^{-1})\cdot s\cdot g,
\end{align*}
such that $g\in B$.
Thus by Conclusion \ref{coset-bn-unique},
$$[u(b)\cdot s\cdot u(-b)]\cdot [u(a)\cdot s]^{k_{1}}\in [u(a)\cdot s]^{k_{2}}\cdot [u(b)\cdot s\cdot u(-b)]B,$$
 if and only if
\begin{equation}\label{a-k-1-k-2}
b-[a_{k_{1}}-b]^{-1}=b_{k_{2}}.
\end{equation}
Now, by using Theorem \ref{ab-dickson}, Definition \ref{beta}, and Proposition \ref{alpha-recursive}, the following holds:
\begin{align*}
b-b_{k_{2}}&=b-\frac{\beta_{k_{2}}(a,b)}{\beta_{k_{2}-1}(a,b)} \\ &= b-\frac{b\cdot\alpha_{k_{2}}(a)-\alpha_{k_{2}-1}(a)}{b\cdot\alpha_{k_{2}-1}(a)-\alpha_{k_{2}-2}(a)}\\ &= \frac{(b^{2}+1)\cdot \alpha_{k_{2}-1}(a)-b\cdot [\alpha_{k_{2}-2}(a)+\alpha_{k_{2}}(a)]}{b\cdot\alpha_{k_{2}-1}(a)-\alpha_{k_{2}-2}(a)}  \\ &= \frac{(b^{2}-a\cdot b+1)\cdot \alpha_{k_{2}-1}(a)}{b\cdot\alpha_{k_{2}-1}(a)-\alpha_{k_{2}-2}(a)}.
\end{align*}
Hence, by using Theorem \ref{ab-dickson} and Proposition \ref{prod-sum-a}, the following holds:
\begin{align*}
\left(b-b_{k_{2}}\right)^{-1} &=\frac{b-a_{k_{2}-1}^{-1}}{b^{2}-a\cdot b+1} \\ &= \frac{b-a_{t-k_{2}}}{b^{2}-a\cdot b+1} \\ &= \frac{b-a+a_{k_{2}}}{b^{2}-a\cdot b+1}.
\end{align*}
Therefore,
by Equation \ref{a-k-1-k-2}, we conclude:
\begin{align*}
a_{k_{1}}-b=\left(b-b_{k_{2}}\right)^{-1}&=\frac{b-a_{t-k_{2}}}{b^{2}-a\cdot b+1} \\&=\frac{b-a+a_{k_{2}}}{b^{2}-a\cdot b+1} \\&=\frac{b-a+a_{k_{2}}}{P_{a}(b)}.
\end{align*}

Hence,
$$[u(b)\cdot s\cdot u(-b)]\cdot [u(a)\cdot s]^{k_{1}}\in [u(a)\cdot s]^{k_{2}}\cdot [u(b)\cdot s\cdot u(-b)]B$$
if and only if
$$\frac{b-a_{t-k_{2}}}{a_{k_{1}}-b}=\frac{b-a+a_{k_{2}}}{a_{k_{1}}-b}=P_{a}(b).$$

Now. consider the case, where $b^{2}-a\cdot b+2=0$. Then, by the first part of the Theorem we have:
$$-1=P_{a}(b)=\frac{b-a_{t-k_{2}}}{a_{k_{1}}-b},$$
which implies:
$$b-{a_{k_1}}=b-a_{t-k_{2}}.$$
Therefore, the theorem holds.
\end{proof}
\\

The next corollary shows an interesting application of Theorem \ref{b-a-a-b}, in the special case where $b=\frac{a}{2}$ (The only case of $b$, where there is no $r\neq b$ such that $P_{a}(b)=P_{a}(r)$). Since in that case by Proposition \ref{chebyshev}, for every $0\leq k\leq t-1$, $\alpha_{k}(a)=U_{k}(b)$ (The Chebyshev polynomial of the second kind on variable $b$) and $\beta_{k-1}(a,b)=T_{k}(b)$ (The Chebyshev polynomial of the first kind on variable $b$), we conclude by Theorem \ref{ab-dickson} that $a_{k}=\frac{U_{k}(b)}{U_{k-1}(b)}$ and $b_{k}=\frac{T_{k+1}(b)}{T_{k}(b)}$ (i.e., both $a_{k}$ and $b_{k}$ are quotients of Chebyshev polynomials on variable $b$).

\begin{corollary}\label{b-half-a}
Let $q=p^{n}$, where $p$ is an odd prime, $n\in \mathbb{N}$ and $4 | q-1$.  Assume  $a=2b$ (i.e., $b=\frac{a}{2}$).
Then for every $k_{1}$ and $k_{2}$ such that $0\leq k_{1}, k_{2}\leq t-1$ the following holds:
$$[u(b)\cdot s\cdot u(-b)]\cdot [u(a)\cdot s]^{k_{1}}\in [u(a)\cdot s]^{k_{2}}\cdot [u(b)\cdot s\cdot u(-b)]B$$
if and only if
$$\frac{a_{k_{2}}-b}{a_{k_{1}}-b}=P_{a}(b)=1-b^{2}.$$
\end{corollary}

\begin{proof}
Assume, $q=p^{n}$, where $p$ is an odd prime, $n\in \mathbb{N}$ and $4 | q-1$. By Theorem \ref{main}, the order of the element $u(a)\cdot s$ is $t=\frac{q+1}{2}$ (which is odd, since $4 | q-1$) in $PSL_{2}(q)$. By proposition \ref{prod-sum-a}, $a_{t-k}=a-a_{k}$ for every $1\leq k\leq t-1$, and  by Proposition \ref{ab-unique}, the $t-1$ elements of the form $a_{k}$ are all different. Hence, there is no $k$ such that $1\leq k\leq t-1$, and $a_{k}=\frac{a}{2}$. Hence,  it is possible to choose $b$ such that $a=2b$ and the left coset $(u(b)\cdot s)B$ is different from  all left cosets of the form  $[u(a)\cdot s]^{k}B$ for $0\leq k<\frac{q+1}{2}$.
Hence, the conditions of Theorems \ref{main} and \ref{b-a-a-b} holds, and by applying Theorem \ref{b-a-a-b} for the case of $a=2b$, we conclude the result of the corollary.
\end{proof}
\\

Proposition \ref{chebyshev} and Corollary \ref{b-half-a} demonstrates interesting properties of the $OGS$ for $PSL_{2}(q)$, where $q$ is odd and there possible to choose $a, b\in \mathbb{F}_{q}$, such that $b=\frac{a}{2}$. By Corollary \ref{a-middle}, $a_{\frac{q+1}{4}}=\frac{a}{2}$, in case of $\mathbb{F}_{q}$, where $4 | q+1$. Therefore, by Proposition \ref{ab-unique}, there is no possibility to choose $a, b\in \mathbb{F}_{q}$ such that $b=\frac{a}{2}$, but the next corollary demonstrates that the choice of $a,b\in\mathbb{F}_{q}$, such that $b^{2}-a\cdot b+2=0$, leads to interesting properties of the defined $OGS$ in Theorem \ref{main}, for the case where $4 |q+1$.

\begin{corollary}\label{4-not-divide}
Let $q=p^{n}$, where $p$ is an odd prime, $n\in \mathbb{N}$ and $4 | q+1$. If $b$ can be chosen such that $P_{a}(b)=-1$ (i.e., the polynomial $\lambda^{2}-a\cdot\lambda+2$ is reducible over $\mathbb{F}_{q}$), then for every $0\leq k\leq t-1$ we have:
$$[u(b)\cdot s\cdot u(-b)]\cdot [u(a)\cdot s]^{k}\in [u(a)\cdot s]^{-k}\cdot [u(b)\cdot s\cdot u(-b)]B.$$
\end{corollary}

\begin{proof}
By Proposition \ref{alpha-property}, $P_{a}(a_{k})=\alpha_{k-1}^{-2}$. Since $4 | q+1$, there is no element $x\in \mathbb{F}_{q}$, such that $x^{2}=-1$. Hence, for every $0\geq k\leq t-2$, $\alpha_{k-1}^{-2}\neq -1$. Thus there is no $a_{k}$ such that $P_{a}(a_{k})=a_{k}^{2}-a\cdot a_{k}+1=-1$. But since the polynomial $\lambda^{2}-a\cdot\lambda+2$ is reducible over $\mathbb{F}_{q}$ reducible over $\mathbb{F}_{q}$, its root can be chosen as $b\neq a_{k}$ for any $k$ such that $0\leq k\leq t-2$. Then by Theorem \ref{b-a-a-b}, we get the result of the corollary.
\end{proof}

\begin{remark}\label{fib3}
In contrast to Corollary \ref{b-half-a}, which is applicable in every $\mathbb{F}_{q}$ such that $4 | q-1$, to apply Corollary \ref{4-not-divide}, we need an element $a\in \mathbb{F}_{q}$ such that all the following holds
\begin{itemize}
\item The polynomial $P_{a}(\lambda)$ is an irreducible polynomial over $\mathbb{F}_{q}$;
\item The root $\omega$ of the polynomial $P_{a}(\lambda)$ has an order $q+1$ over $\mathbb{F}_{q^{2}}$;
\item The polynomial $P_{a}(\lambda)+1$ (i.e., $\lambda^{2}-a\cdot\lambda+2$) is a reducible polynomial over $\mathbb{F}_{q}$ (otherwise there is no possibility for a choice of $b\in \mathbb{F}_{q}$ such that $P_{a}(b)=-1$).
\end{itemize}
There are finite fields $\mathbb{F}_{q}$, where $4 | q+1$, but any $a\in \mathbb{F}_{q}$ does not satisfied the three mentioned conditions above (e.g., $\mathbb{F}_{11}$). The case where Corollary \ref{4-not-divide} is applicable for $a=3$ is an interesting case, since then the choice of $b=1$ is applicable as well ($b=1$ is a root of the polynomial $\lambda^{2}-3\cdot \lambda+2)$ for satisfying Corollary \ref{4-not-divide}. Then by Proposition \ref{fibonacci}, both $\alpha_{k}$ and $\beta_{\ell}$ are Fibonacci numbers over $\mathbb{F}_{q}$ for $1\leq k\leq t-1$ and $0\leq \ell\leq t-1$, and by Theorem \ref{ab-dickson}
$$a_{k}=\frac{Fib_{2k+2}}{Fib_{2k+1}} ~~~~ b_{k}=\frac{Fib_{2k+1}}{Fib_{2k}}.$$
\end{remark}

The next theorem demonstrates very interesting connections between the $BN-pair$ presentation of the element $[u(a)\cdot s]^{k}$ of $PSL_{2}(q)$ to $\alpha_{r}(a)$ (The Dickson polynomial of the second kind on variable $a$ over $\mathbb{F}_{q}$), and the presentation of  $[u(a)\cdot s]^{\ell}\cdot [u(b)\cdot s\cdot u(-b)]$ to $\beta_{r}(a,b)$, and $\gamma_{r}(a,b)$ (Where by Definition \ref{beta}, both $\beta_{r}(a,b)$, and $\gamma_{r}(a,b)$ are linear combinations of $\alpha_{r}(a)$ and $\alpha_{r-1}(a)$, and in case $b=\frac{a}{2}$, by Proposition \ref{chebyshev}, $\beta_{r}(a,b)$ is the Chebyshev polynomial of the first kind on variable $b$ over $\mathbb{F}_{q}$), for $1\leq k\leq t-1$, $0\leq \ell\leq t-1$, $-1\leq r\leq t-1$ and $t=\frac{q+1}{gcd(2, q+1)}$.

\begin{theorem}\label{x-k-y-k}
Let $G=PSL_{2}(q)$. For every  $1\leq k\leq t-1$, $0\leq \ell\leq t-1$, and  $-1\leq r\leq t-1$, let $a_{k}$ and $b_{\ell}$ be elements of $\mathbb{F}_{q}$ as it is defined in Proposition \ref{ab-unique}  and let $\alpha_{r}(a)$,  $\beta_{r}(a,b)$, and $\gamma_{r}(a,b)$ be elements of $\mathbb{F}_{q}$ as it is defined in Definitions \ref{alpha} and \ref{beta}. Then

$$[u(a)\cdot s]^{k}=u(a_{k})\cdot s\cdot u(x_{k})\cdot h(y_{k})$$
 and
 $$[u(a)\cdot s]^{\ell}\cdot [u(b)\cdot s\cdot u(-b)]=u(b_{\ell})\cdot s\cdot u({x}'_{\ell})\cdot h({y}'_{\ell}),$$
such that:
\begin{itemize}
\item $x_{k}=-\alpha_{k-2}(a)\cdot \alpha_{k-1}(a)$;
\item $y_{k}=\alpha_{k-1}(a)$;
\item ${x}'_{\ell}=-\beta_{\ell-1}(a,b)\cdot \gamma_{\ell-1}(a,b)$;
\item ${y}'_{\ell}=\beta_{\ell-1}(a,b)$;

\end{itemize}
\end{theorem}

\begin{proof}
We prove the results of the theorem for $x_{k}$, $y_{k}$, ${x}'_{k}$, and for ${y}'_{k}$ by induction on $k$. Recall, by Proposition \ref{ab-unique}, $$[u(a)\cdot s]^{k}=u(a_{k})\cdot s\cdot u(x_{k})\cdot h(y_{k})$$ and $$[u(a)\cdot s]^{k}\cdot[u(b)\cdot s\cdot u(-b)]=u(b_{k})\cdot s\cdot u({x}'_{k})\cdot h({y}'_{k}).$$ By substituting $k=1$, we have
\begin{align*}
u(a)\cdot s &= u(a)\cdot s\cdot u(0)\cdot h(1) \\ &= u(a_{1})\cdot s\cdot u(-\alpha_{-1}(a)\cdot \alpha_{0}(a))\cdot h(\alpha_{0}(a)) \\ &= u(a_{1})\cdot s\cdot u(x_{1})\cdot h(y_{1}).
\end{align*}
Thus, $x_{1}=-\alpha_{-1}(a)\cdot \alpha_{0}(a)$ and $y_{1}=\alpha_{0}(a)$. \\
By using the multiplication laws which we have described in Proposition \ref{uhs-multiply},  we have also,
\begin{align*}
u(a)\cdot s\cdot u(b)\cdot s\cdot u(-b) &= u(a-b^{-1})\cdot s\cdot u(-b)\cdot h(b)\cdot u(-b) \\ &= u(b_{1})\cdot s\cdot u(-b-b^{3})\cdot h(b) \\ &= u(b_{1})\cdot s\cdot u(-b\cdot[1+b^2])\cdot h(b) \\ &= u(b_{1})\cdot s\cdot u(-\beta_{0}(a,b)\cdot \gamma_{0}(a,b))\cdot h(\beta_{0}(a,b)) \\ &= u(b_{1})\cdot s\cdot u({x}'_{1})\cdot h({y}'_{1}).
\end{align*}
Thus,
$${x}'_{0}=-\beta_{-1}(a,b)\cdot [\alpha_{-1}(a)+b\cdot \beta_{-1}(a,b)]$$ and
$${y}'_{0}=\beta_{-1}(a,b).$$
Now, assume by induction
$$x_{r}=-\alpha_{r-2}(a)\cdot \alpha_{r-1}(a),$$
$$y_{r}=\alpha_{r-1}(a),$$
$${x}'_{r}=-\beta_{r-1}(a,b)\cdot [\alpha_{r-1}(a)+b\cdot \beta_{r-1}(a,b)],$$
and
$${y}'_{r}=\beta_{r-1}(a,b)$$
for $r\leq k$, and we prove the correctness of the proposition for $r=k+1$.
By using our induction hypothesis, Proposition \ref{uhs-multiply}, and the identity
$$\alpha^{2}_{k-1}(a)=1+\alpha_{k}(a)\cdot \alpha_{k-2}(a)$$
from the last part of Proposition \ref{alpha-property}, the following holds:
\begin{align*}
[u(a)\cdot s]^{k+1} &= u(a)\cdot s\cdot [u(a)\cdot s]^{k} \\ &= u(a)\cdot s\cdot u(a_{k})\cdot s\cdot u(-\alpha_{k-2}(a)\cdot \alpha_{k-1}(a))\cdot h(\alpha_{k-1}(a)) \\ &= u(a-{a_{k}}^{-1})\cdot s\cdot u(-a_{k})\cdot h(a_{k})\cdot u(-\alpha_{k-2}(a)\cdot \alpha_{k-1}(a))\cdot h(\alpha_{k-1}(a)) \\ &= u(a_{k+1})\cdot s\cdot u\left(-\frac{\alpha_{k}(a)}{\alpha_{k-1}(a)}\right)\cdot h\left(\frac{\alpha_{k}(a)}{\alpha_{k-1}(a)}\right)\cdot \\ &\cdot u(-\alpha_{k-2}(a)\cdot \alpha_{k-1}(a))\cdot h(\alpha_{k-1}(a)) \\ &= u(a_{k+1})\cdot s\cdot u\left(-\frac{\alpha_{k}(a)+\alpha^{2}_{k}(a)\cdot \alpha_{k-2}(a)}{\alpha_{k-1}(a)}\right)\cdot h(\alpha_{k}(a)) \\ &= u(a_{k+1})\cdot s\cdot u\left(-\frac{\alpha_{k}(a)\cdot[1+\alpha_{k}(a)\cdot \alpha_{k-2}(a)]}{\alpha_{k-1}(a)}\right)\cdot h\left(\alpha_{k}(a)\right) \\ &= u(a_{k+1})\cdot s\cdot u(-\alpha_{k}(a)\cdot \alpha_{k-1}(a))\cdot h(\alpha_{k}(a)).
\end{align*}
Thus,
$$x_{k}=-\alpha_{k-2}(a)\cdot \alpha_{k-1}(a),$$
$$y_{k}=\alpha_{k-1}(a),$$
for every $1\leq k\leq t-1$. \\
Similarly,
\begin{align*}
[u(a)\cdot s]^{k+1}\cdot [u(b)\cdot s\cdot u(-b)] &= u(a)\cdot s\cdot [u(a)\cdot s]^{k}\cdot [u(b)\cdot s\cdot u(-b)] \\ &= u(a)\cdot s\cdot u(b_{k})\cdot s\cdot u(-\beta_{k-1}(a,b)\cdot \gamma_{k-1}(a,b))\cdot h(\beta_{k-1}(a,b)) \\ &= u(a-{b_{k}}^{-1})\cdot s\cdot u(-b_{k})\cdot h(b_{k})\cdot \\ &\cdot u(-\beta_{k-1}(a,b)\cdot \gamma_{k-1}(a,b))\cdot h(\beta_{k-1}(a,b)) \\ &= u(b_{k+1})\cdot s\cdot u\left(-\frac{\beta_{k}(a,b)}{\beta_{k-1}(a,b)}\right)\cdot h\left(\frac{\beta_{k}(a,b)}{\beta_{k-1}(a,b)}\right)\cdot  \\ &\cdot u(-\beta_{k-1}(a,b)\cdot \gamma_{k-1}(a,b))\cdot h(\beta_{k-1}(a,b)) \\ &= u(b_{k+1})\cdot s\cdot \\ &\cdot u\left(-\frac{\beta_{k}(a,b)\cdot [1+\beta_{k}(a,b)\cdot \gamma_{k-1}(a,b)]}{\beta_{k-1}(a,b)}\right)\cdot \\ &\cdot h\left(\beta_{k}(a,b)\right) \\ &= u(b_{k+1})\cdot s\cdot u(-\beta_{k}(a,b)\cdot\gamma_{k}(a,b))\cdot h(\beta_{k}(a,b)).
\end{align*}
Thus,
\begin{align*}
{x}'_{k} &= -\beta_{k-1}(a,b)\cdot [\alpha_{k-1}(a)+b\cdot \beta_{k-1}(a,b)] \\ &= -\beta_{k-1}(a,b)\cdot\gamma_{k-1}(a,b),
\end{align*}
and
$${y}'_{k}=\beta_{k-1}(a,b),$$
for every $0\leq k\leq t-1$.
\end{proof}

\begin{proposition}\label{x-y-ogs-bn-pair}
Let $G=PSL_{2}(q)$.  For every $1\leq k\leq t-1$, $0\leq \ell\leq t-1$, let $a_{k}, b_{\ell}$ be elements of $\mathbb{F}_{q}$ as it is defined in Proposition \ref{ab-unique}, and for $-1\leq r\leq t-1$ let  $\alpha_{r}(a)$,  $\beta_{r}(a,b)$, and $\gamma_{r}(a,b)$ be elements of $\mathbb{F}_{q}$ as it is defined in Definition \ref{alpha} and \ref{beta}. Then, the following holds:
\begin{itemize}
\item
\begin{align*}
&[u(a)\cdot s]^{k}\cdot u(x)\cdot h(y) \\ &=u(a_{k})\cdot s\cdot u(\alpha_{k-1}(a)\cdot[\alpha_{k-1}(a)\cdot x-\alpha_{k-2}(a)])\cdot h(\alpha_{k-1}(a)\cdot y)\\ &=u(a_{k})\cdot s\cdot u(\alpha_{k-1}^{2}(a)\cdot[x-a_{t-k}])\cdot h(\alpha_{k-1}(a)\cdot y)\\ &=u(a_{k})\cdot s\cdot u(\alpha_{k-1}^{2}(a)\cdot[x+a_{k}-a])\cdot h(\alpha_{k-1}(a)\cdot y);
\end{align*}
\item
\begin{align*}
&[u(a)\cdot s]^{\ell}\cdot [u(b)\cdot s\cdot u(-b)]\cdot u(x)\cdot h(y) \\ &=u(b_{\ell})\cdot s\cdot u(\beta_{\ell-1}(a,b)\cdot[\beta_{\ell-1}(a,b)\cdot x-\gamma_{\ell-1}(a,b)])\cdot h(\beta_{\ell-1}(a,b)\cdot y).
\end{align*}
\end{itemize}
\end{proposition}

\begin{proof}
The proof is straightforward using of Theorems \ref{x-k-y-k}, \ref{ab-dickson}, and Proposition \ref{uhs-multiply}.
\end{proof}

\begin{corollary}\label{x-0}
Let $g\in PSL_{2}(q)$, such that the matrix presentation of $\tilde{g}$ is as follows: $\tilde{g}=\begin{pmatrix} r_{1} & -r_{2}^{-1} \\
 r_{2} & 0\end{pmatrix}$ for some $r_{1}\in \mathbb{F}_{q}$ and $r_{2}\in \mathbb{F}_{q}^{*}$. Let $a_{k}\in \mathbb{F}_{q}$ as it is defined in Proposition \ref{ab-unique}, $\alpha_{k}(a)$,  $\beta_{k}(a,b)$, and $\gamma_{k}(a,b)$ as it is defined in Definition \ref{alpha} and \ref{beta}. Then,  $g=u(\tilde{a})\cdot s\cdot h(\tilde{y})$, such that $\tilde{y}=r_{2}$, $\tilde{a}=\frac{r_{1}}{r_{2}}$, and one of the following holds:
\begin{itemize}
\item If $\tilde{a}=a_{k}$ for $1\leq k<\frac{q+1}{gcd(2, q+1)}$, then $g=[u(a)\cdot s]^{k}\cdot u(x)\cdot h(y)$, such that  $x=a_{t-k}=a-a_{k}$ ($t=\frac{q+1}{gcd(2, q+1)}$);
\item If $\tilde{a}=b_{k}$ for $0\leq k<\frac{q+1}{gcd(2, q+1)}$ and $q$ is odd, then \\ \\ $g=[u(a)\cdot s]^{k}\cdot [u(b)\cdot s\cdot u(-b)]\cdot u(x)\cdot h(y)$, such that \\ \\$x=\frac{\gamma_{k-1}(a,b)}{\beta_{k-1}(a,b)}=b+\frac{\alpha_{k-1}(a)}{\beta_{k-1}(a,b)}$.
\end{itemize}
\end{corollary}

\begin{proof}
The proof is a direct consequence of Proposition \ref{x-y-ogs-bn-pair}.
\end{proof}

\section{Conclusions and future plans}

The results of the paper provide intriguing connections between the $BN-pair$ presentation and the $OGS$ presentation of $PSL_{2}(q)$, which is a family of finite simple groups, where $q\geq 4$. These connections allow to studying very interesting properties of some important recursive sequences over a finite field $\mathbb{F}_{q}$, which are closely connected to the Dickson polynomials and the Chebyshev polynomials of the second kind (in special cases to the Dickson and the Chebyshev polynomials of the first kind as well). The results motivate us to further research in open questions about the connections between the mentioned presentations ($BN-pair$ and $OGS$) of $PSL_{n}(q)$ for $n>2$ and other families of simple groups of Lie-type.


\begin{thebibliography}{99}
\bibitem{szep} J. Sz\'{e}p, "On the structure of groups which can be represented as the product of two subgroups", \emph{Acta Sci. Math. Szeged} \textbf{12}  (1950)
      57-61.
\end{thebibliography}
\end{document}